\documentclass[12pt]{amsart}
\usepackage{amssymb}
\usepackage{amsmath}
\usepackage{versions}
\usepackage{amsfonts}
\usepackage[usenames]{color}
\usepackage{graphicx}
 \newtheorem{theorem}{Theorem}[section]
 \newtheorem{corollary}[theorem]{Corollary}
 \newtheorem{lemma}[theorem]{Lemma}
 \newtheorem{proposition}[theorem]{Proposition}

\newtheorem{definition}[theorem]{Definition}

\newtheorem{example}[theorem]{Example}

\newtheorem{fact*}{Fact}

\DeclareMathOperator{\diag}{diag}
\DeclareMathOperator\re{Re}

\DeclareMathOperator\spa{span}

\DeclareMathOperator\hol{Hol}
\DeclareMathOperator\aut{Aut}
\newcommand\half{{\tfrac 12}}
\newcommand\quart{\tfrac 14}

\newcommand\id{\mathrm{id}}

\newcommand{\eps}{\varepsilon}

\newcommand{\PP}{\mathcal{P}}

\newcommand{\T}{\mathbb{T}}

\newcommand{\D}{\mathbb{D}}
\newcommand{\C}{\mathbb{C}}
\newcommand{\mat}{\mathbb{C}^{2\times 2}}
\newcommand{\ccc}{\mathcal{P}}
\newcommand\G{\mathcal {G}}

\newcommand{\schur}{\mathcal{S}}
\newcommand{\R}{\mathbb{R}}

\newcommand\Aut{\mathrm{Aut}~}

\newcommand{\tr}{\operatorname{tr}}

\newcommand{\inv}{^{-1}}

\newcommand{\ph}{\varphi}
\newcommand\al{\alpha}

\newcommand\la{\lambda}
\newcommand\La{\Lambda}
\newcommand\up{\upsilon}
\newcommand\ups{\upsilon}
\newcommand\si{\sigma}
\newcommand\beq{\begin{equation}}

\newcommand\eeq{\end{equation}}
\newcommand\df{\stackrel{\rm def}{=}}
\newcommand\ii{\mathrm i}

\newcommand\blue{\color{black}}
\newcommand\black{\color{black}}

\newcommand\nn{\nonumber}
\newcommand\bbm{\begin{bmatrix}}
\newcommand\ebm{\end{bmatrix}}
\numberwithin{equation}{section}

\newcommand\B{\mathbb B}

\begin{document}
\title[A $\mu$-related domain]{The complex geometry of a domain related to $\mu$-synthesis}
\author{Jim Agler, Zinaida A. Lykova and N. J. Young}
\date{27th March 2014, Revised 2nd September 2014}

\begin{abstract} 
We establish the basic complex geometry and function theory of the {\em pentablock} $\mathcal{P}$, which is the bounded domain 
\[
\mathcal{P}= \{(a_{21}, \tr A, \det A): A= \bbm a_{ij}\ebm_{i,j=1}^2 \in \mathbb{B}\}
\]
where $\mathbb{B}$ denotes the open unit ball in the space  of $2\times 2$ complex matrices.  We prove several characterizations of the domain.   We show that $\mathcal{P}$ arises naturally in connection with a certain robust stabilization problem in control theory, 
 the problem of {\em $\mu$-synthesis}.  We describe the distinguished boundary of $\mathcal{P}$ and exhibit a $4$-parameter group of automorphisms of $\mathcal{P}$.     We demonstrate connections between the function theories of $\mathcal{P}$ and $\mathbb{B}$.  We show that $\mathcal{P}$ is polynomially convex and starlike, and we show that the real pentablock $\mathcal{P}\cap \R^3$ is a convex set bounded by five faces, three of them flat and two curved.

\end{abstract}

\thanks{The first author was partially supported by National Science Foundation Grants on  Extending Hilbert Space Operators DMS 1068830 and DMS 1361720. The third author was partially supported by the UK Engineering and Physical Sciences Research Council grant  EP/K50340X/1. The collaboration was partially supported by London Mathematical Society Grant 41321.}
\maketitle

\section*{Contents} \label{contents}

\ref{intro}. Introduction \hfill  Page \pageref{intro}

\ref{Gamma}. The symmetrised bidisc and the pentablock  \hfill \pageref{Gamma}

\ref{instance}. An instance of $\mu$ and an associated domain \hfill \pageref{instance}

\ref{Linfrac}. A class of linear fractional functions \hfill\pageref{Linfrac}

\ref{twodomains}. The domains $\mathcal P$ and $\ccc_\mu$
 \hfill\pageref{twodomains}

\ref{elemGeom}. Elementary geometry of the pentablock  \hfill\pageref{elemGeom}

\ref{automorphisms}. Some automorphisms of $\PP$
\hfill\pageref{automorphisms}

\ref{disting_bound}. The distinguished boundary of $\PP$ \hfill\pageref{disting_bound}

\ref{real_pentablock}. The real pentablock $\ccc\cap \R^3 $ \hfill \pageref{real_pentablock}

\ref{schwarz_mu}. A Schwarz Lemma for a general $\mu$  \hfill\pageref{schwarz_mu}

\ref{schwarz}. What is the Schwarz Lemma for $\PP$? \hfill\pageref{schwarz}

\ref{lifting}. Analytic lifting \hfill\pageref{lifting}

\ref{conclud}. Conclusions \hfill\pageref{conclud}

 References \hfill\pageref{pentablock_bbl}

\section{Introduction}\label{intro}
In this paper  we establish the basic complex geometry and function theory of the domain
\beq\label{defPent}
\ccc= \{(a_{21}, \tr A, \det A): A= \bbm a_{ij}\ebm_{i,j=1}^2 \in \B\}
\eeq
where $\B$ denotes the open unit ball in the space $\C^{2\times 2}$ of $2\times 2$ complex matrices, with the usual operator norm. 
  We call this domain the {\em pentablock}.  The name alludes to the fact that $\ccc\cap\R^3$ is a convex body bounded by five faces, three of them flat and two curved (Theorem \ref{PcapR3}).   $\ccc$ is a holomorphic image of the Cartan domain $\B$.  It is polynomially convex and starlike about the origin, but neither circled nor convex.  The paper contains several characterizations of the domain,  and descriptions of its distinguished boundary and of a $4$-parameter group of automorphisms  and of connections with the function theory of $\B$. 

The domain $\ccc$ arises in connection with the {\em structured singular value}, a cost function on matrices introduced by control engineers in the context of robust stabilization with respect to modelling uncertainty \cite{Do}.  The structured singular value is denoted by $\mu$, and engineers have proposed an interpolation problem called the {\em $\mu$-synthesis problem} that arises from this source.  Attempts to solve cases of this interpolation problem have led to the study of two other domains, the {\em symmetrised bidisc} \cite{AY04} and the {\em tetrablock} \cite{awy}, in $\C^2$ and $\C^3$ respectively, which have turned out to have many properties of interest to specialists in several complex variables \cite{NiPfZw,EZ,EdKoZw10} and to operator theorists \cite{BPSR,Sarkar}.  The relationship between $\ccc$ and an instance of $\mu$ is explained in Section \ref{twodomains}, and there is a more thoroughgoing discussion in the Conclusions (Section \ref{conclud}).

We shall denote the open unit disc by $\D$, its closure by $\Delta$ and the unit circle by $\T$.  The polynomial map implicit in the definition \eqref{defPent} will be written
\beq\label{defpi}
\pi(A)=(a_{21}, \tr A, \det A) \quad \mbox{ where } A=\bbm a_{ij} \ebm_{i,j=1}^2\in\mat.
\eeq
Thus $\ccc= \pi(\B)$.  For the $\mu$ in question it transpires that $\mu(A)<1$ if and only if $\pi(A)\in\ccc$.   This statement is contained in Theorem \ref{4equiv}, one of the main results of the paper.   To illustrate the flavour of our results, here are foretastes of Theorem \ref{4equiv} and Theorem \ref{aut_PP}.
\begin{theorem}\label{foretaste}
Let
\[
(s,p)= (\la_1+\la_2,\la_1\la_2)
\]
where $\la_1,\la_2 \in\D$.  Let $a\in\C$ and let
\[
\beta=\frac{s-\bar sp}{1-|p|^2}.
\]
The following statements are equivalent.
\begin{enumerate}
\item $(a,s,p)\in\PP$;
\item there exists $A\in\mat$ such that $\mu(A) < 1$ and $\pi(A)=(a,s,p)$;
\item $|a|<\left|1- \frac{\half s\bar\beta}{1+\sqrt{1-|\beta|^2}}\right|$;
\item $|a| < \half|1-\bar\la_2\la_1|+\half (1-|\la_1|^2)^\half(1-|\la_2|^2)^\half$;
\item $\sup_{z\in\D} \left|\Psi_z(a,s,p)\right|  < 1$.
\end{enumerate}
\end{theorem}
In this statement the cost function $\mu$ on $\mat$ is defined in Section \ref{instance}, and $\Psi_z$ is the linear fractional map
\[
\Psi_z(a,s,p)= \frac{a(1-|z|^2)}{1-sz+pz^2}.
\]
The significance of the equivalence of (1) and (2) is explained in the concluding section.

\begin{theorem}\label{foretaste2}
For every $\omega\in\T$  and every automorphism $\up$ of $\D$, the map
\beq \label{automP2}
f_{\omega \up}( a, \la_1+\la_2,\la_1\la_2)= 
\left( \frac{\omega \eta(1-|\alpha|^2) a}{1-  \bar\alpha (\la_1+\la_2)+ \bar{\alpha}^2 \la_1\la_2}, \up(\la_1)+\up(\la_2), \up(\la_1)\up(\la_2) \right)
\eeq
is an automorphism of $\ccc$, where 
\[
\up(\la) = \eta \frac{\la-\alpha}{1-\bar\al \la}
\]
for some $\eta\in\T$ and $\al\in\D$.  The maps $\{f_{\omega \up}: \omega\in\T, \up\in\aut  \D\}$ comprise a group of automorphisms of $\ccc$.
\end{theorem}

\section{The symmetrised bidisc and the pentablock}\label{Gamma}

The pentablock is closely related to the symmetrised bidisc, which is the domain
\beq\label{defG}
\G= \{(z+w,zw): |z|<1, \ |w| < 1\}
\eeq
in $\C^2$.  Indeed, it is clear from the definition \eqref{defPent} that $\ccc$ is fibred over $\G$ by the map $(a,s,p)\mapsto (s,p)$, since if $A\in\B$ then the eigenvalues of $A$ lie in $\D$ and so $(\tr A,\det A) \in\G$.

Some basic properties of $\G$ will be needed, in particular the following characterizations \cite{AY04}.
\begin{theorem}\label{recapG}
For a point $(s,p) \in \C^2$ the following statements are equivalent.
\begin{enumerate}
\item $(s,p) \in\G$;
\item $|s-\bar sp| < 1-|p|^2$;
\item $|p|<1$ and there exists $\beta\in\D$ such that $s=\beta+\bar\beta p$;
\item there exists $A\in\B$ such that $\tr A=s$ and $\det A=p$.
\end{enumerate}
\end{theorem}
The following observation will facilitate the construction of matrices in $\B$.
\begin{lemma}\label{easy}
If the eigenvalues of $A\in\C^{2\times 2}$ lie in $\Delta$ then $\|A\| < 1$ if and only if $\det(1-A^*A) > 0$.
\end{lemma}
\begin{proof}
Necessity is clear.  Conversely, suppose that $\si(A)\subset\Delta$ and $\det(1-A^*A) > 0$ but $\|A\| \geq 1$.  Let $A$ have eigenvalues $\la_1, \la_2$ and singular values $s_0,s_1$.  Then $s_0 \geq 1$ and $1-A^*A$ is unitarily equivalent to the matrix $\diag\{1-s_0^2, 1-s_1^2\}$.  Hence
\[
0< \det(1-A^*A) =(1-s_0^2)(1-s_1^2).
\]
Since $1-s_0^2 \leq 0$ it follows that $1-s_1^2 < 0$, that is, $s_0, s_1 > 1$.  Therefore
\[
1<s_0s_1 =|\det A| =|\la_1\la_2| \leq 1,
\]
a contradiction.   Thus $\|A\| < 1$.
\end{proof}
\begin{proposition}\label{asuffCondn}
Let
\beq\label{lambdas}
(s,p) = (\la_1+\la_2,\la_1\la_2)\in\G.
\eeq
If $a\in\C$ satisfies
\beq\label{modacndn}
|a| < \half|1-\bar\la_2\la_1|+\half (1-|\la_1|^2)^\half(1-|\la_2|^2)^\half
\eeq
then $(a,s,p) \in \ccc$.
\end{proposition}

\begin{proof}
Consider $(a,s,p)$ with $(s,p)$ as in equation \eqref{lambdas} and $a$ satisfying the inequality \eqref{modacndn}.  We must construct $A\in\mat$ such that $\|A\|<1, \, a_{21}=a,\, \tr A=s$ and $\det A =p$.   Let 
\[
\Lambda=(1-|\la_1|^2)^\half(1-|\la_2|^2)^\half
\]
and define $c_\pm$ by
\[
c_\pm = \half |1-\bar\la_2\la_1| \pm \half \Lambda.
\]
Note that $0<c_- < c_+$.

Consider the case that $c_- < |a| < c_+$.
Let $w=\half(\la_1-\la_2)$, so that $w^2=\tfrac 14 s^2-p$, and let
\[
A=\bbm \half s& w^2/a \\ a & \half s \ebm.
\]
We have $\tr A=s, \, \det A=p$ and
\begin{align}\label{detpoly}
|a|^2\det(1-A^*A) &= |a|^2 (1-\tr(A^*A)+|\det A|^2)\notag\\
	&= -|a|^4 + (1-\half|s|^2+|p|^2)|a|^2 -|w|^4.
\end{align}
Now
\begin{align}\label{keepthis}
c_-^2+c_+^2 &= \half |1-\bar\la_2\la_1|^2 +\half \Lambda^2 \notag\\
	&=\half\{1-2\re(\bar\la_2\la_1) + |\la_1\la_2|^2+1-|\la_1|^2-|\la_2|^2+ |\la_1\la_2|^2\} \notag\\
	&= 1-\half|s|^2+|p|^2
\end{align}
and
\begin{align}\label{eqccw}
c_- c_+ &= \quart\{ |1-\bar\la_2\la_1|^2-\Lambda^2\} \notag\\
	&=\quart\{1-2\re(\bar\la_2\la_1)+|\la_1\la_2|^2-1+|\la_1|^2+|\la_2|^2 -|\la_1\la_2|^2\} \notag\\
	&= \quart |\la_1-\la_2|^2 = |w|^2.
\end{align}
Comparison with equation \eqref{detpoly} reveals that
\[
|a|^2\det(1-A^*A) = -(|a|^2-c_-^2)(|a|^2-c_+^2).
\]
Hence, when $c_-<|a|<c_+$ we have $\det(1-A^*A) > 0$ and so, by Lemma \ref{easy}, $\|A\| < 1$ and therefore  $(a,s,p) \in \ccc$.

In the case that $|a| \leq |w|$ choose $\zeta\in\T$ such that $\la_1-\la_2=\zeta |\la_1-\la_2|$ and let
\[
A=\bbm \half s+(|w|^2-|a|^2)^\half \zeta & \zeta^2\bar a \\ a & \half s-(|w|^2-|a|^2)^\half\zeta \ebm.
\]
Then $\pi(A)=(a,s,p)$, and a simple calculation shows that
\[
\det(1-A^*A) = (1-|\la_1|^2)(1-|\la_2|^2) > 0
\]
and hence $\|A\| < 1$.

We have shown that $(a,s,p) \in \ccc$ in the cases $c_- < |a| < c_+$ and $|a| \leq |w|$.
The proposition will follow if we can show that
\[
|c_-|\le |w|<|c_+|.
\]
Since $|w|$ is the geometric mean of $c_-$ and $c_+$, by \eqref{eqccw}, this inequality is true. 
Thus
$(a,s,p)\in\PP$ for all $a$ such that $|a| < \half|1-\bar\la_2\la_1|+ \half \La$.
\end{proof}
The converse of Proposition \ref{asuffCondn} is also true (Theorem \ref{4equiv}).  Thus the fibre of $\ccc$ over the point $(\la_1+\la_2, \la_1\la_2)$ is the open disc of radius 
\[
\half|1-\bar\la_2\la_1|+\half (1-|\la_1|^2)^\half(1-|\la_2|^2)^\half.
\]

The closure $\bar\ccc$ of $\ccc$ will also play a role; call it the {\em closed pentablock}.   It is elementary that $\bar\ccc$ is the image of the closure $\bar\B$ of $\B$ under $\pi$.  

We denote by $\Gamma$ the closure of $\G$ in $\C^2$, so that
\[
\Gamma = \{ (z+w,zw): |z|\leq 1, |w|\leq 1\}.
\]
\begin{proposition}\label{anothersuff}
Let
\beq\label{morelambdas}
(s,p) = (\la_1+\la_2,\la_1\la_2)\in\Gamma.
\eeq
If $a\in\C$ satisfies
\beq\label{modacndn2}
|a| \leq \half|1-\bar\la_2\la_1|+\half (1-|\la_1|^2)^\half(1-|\la_2|^2)^\half
\eeq
then $(a,s,p) \in \bar\ccc$.
\end{proposition}

\begin{proof}
Let the relations \eqref{morelambdas} and \eqref{modacndn2} hold.  Pick $r\in(0,1)$; then
\[
r|a| \leq \half r|1-\bar\la_2\la_1|+\half r(1-|\la_1|^2)^\half(1-|\la_2|^2)^\half.
\]
   Simple calculations show that
\begin{align*}
r|1-\bar\la_2\la_1| &< |1-r^2\bar\la_2\la_1|, \\
 r(1-|\la_1|^2)^\half(1-|\la_2|^2)^\half &< (1-r^2|\la_1|^2)^\half(1-r^2|\la_2|^2)^\half.
\end{align*}
Hence 
\[
r|a| < \half |1-r^2\bar\la_2\la_1|+\half (1-r^2|\la_1|^2)^\half(1-r^2|\la_2|^2)^\half.
\]
It follows from Proposition \ref{asuffCondn} that $(ra, rs, r^2p) \in \ccc$ for all $r\in (0,1)$.  Hence $(a,s,p) \in\bar\ccc$.
\end{proof}

\section{An instance of $\mu$ and an associated domain}\label{instance}

The {\em structured singular value} $\mu_E$ of $A\in\C^{m\times n}$ corresponding to subspace $E$ of $\C^{n\times m}$ is defined by
\beq\label{defmu}
\frac{1}{\mu_E(A)} = \inf \{\|X\|: X\in E \mbox{ and } \det (1-AX)=0\}.
\eeq
In the cases that  1) $E$ comprises the whole of $\C^{n\times m}$ and 2) $m=n$ and $E$ consists of the scalar multiples of the identity,  $\mu_E$ is a familiar object, to wit 
the operator norm and the spectral radius respectively.   When $E$ comprises the diagonal matrices, $\mu_E$ is an intermediate cost function $\mu_{\mathrm {diag}}$.  In these three cases  the corresponding $\mu$-synthesis problem leads to the analysis of the classical Nevanlinna-Pick interpolation problem, the symmetrised polydisc and (when $m=n=2$) the tetrablock respectively.  In this paper we are concerned with the case that $m=n=2$ and
\[
E=\spa\left\{1,\bbm 0&1\\0&0 \ebm\right\} \subset \C^{2\times2},
\]
another natural choice of $E$.  Observe that a matrix $X=\bbm z&w\\0&z \ebm \in E$ is a contraction if and only if $|w|\leq 1-|z|^2$.

\begin{proposition}\label{mucriter}
For any matrix $A= \bbm a_{ij} \ebm\in\mat$,
\beq\label{formmucriter}
\mu_E(A) < 1 \mbox{ if and only if } (s,p)\in\G \;\mbox{ and }\; |a_{21}|\sup_{z\in\D}\frac{1-|z|^2}{|1-sz+pz^2|}< 1
\eeq
and 
\beq\label{mucriterbis}
\mu_E(A) \leq 1 \mbox{ if and only if } (s,p)\in\Gamma \mbox{ and }\; \frac{|a_{21}|(1-|z|^2)}{|1-sz+pz^2|}\leq 1 \mbox{ for all } z\in\D,
\eeq
where $s=\tr A$ and $ p=\det A$.
\end{proposition}
\begin{proof} 
For $ X=\bbm z&w\\0&z \ebm$, 
\[
1-AX=\bbm 1-a_{11}z&-a_{11}w-a_{12}z \\ -a_{21}z & 1-a_{21}w-a_{22}z \ebm
\]
and so 
\begin{align*}
\det(1-AX)&= 1-(\tr A)z +(\det A)z^2 - a_{21}w\\
	&= 1-sz+pz^2 -a_{21}w.
\end{align*}
We have
\begin{align}\label{mu<1}
\mu_E(A)<1 &\Leftrightarrow  \inf \{\|X\|: X\in E \mbox{ and } \det (1-AX)=0\} > 1.
\end{align}
Suppose that $\mu_E(A) < 1$.  It follows from the last equivalence that if $|w| \leq 1-|z|^2$ then the contraction 
$ X=\bbm z&w\\0&z \ebm$ satisfies $\det(1-AX) \neq 0$, that is, 
\beq\label{spw}
1-sz+pz^2 \neq a_{21} w \quad \mbox{ whenever } |w| \leq 1-|z|^2.
\eeq
  In particular, on taking $w=0$, we find that $1-sz+pz^2 \neq 0$ for all $z\in \Delta$, which is to say that $(s,p)\in\G$.  Furthermore, the inequation \eqref{spw} implies that
\[
|1-sz+pz^2| > |a_{21}| (1-|z|^2) \quad \mbox{ for all } z\in\Delta.
\]
In particular, $|1-sz+pz^2|$ is strictly positive on $\T$, and consequently the function
\[
|1-sz+pz^2|/(1-|z|^2)
\] 
tends to $\infty$ as $ |z| \to 1$ and hence attains its infimum over $\D$ at a point of $\D$.
Necessity in the statement \eqref{formmucriter} follows.

Conversely, suppose that $(s,p)\in\G$ and 
\beq\label{ass}
|a_{21}|\sup_{z\in\D}\frac{1-|z|^2}{|1-sz+pz^2|}< 1.
\eeq
In particular, on letting $z=0$, we have
\beq\label{moda<1}
|a_{21}| < 1.
\eeq
We wish to show that $\mu_E(A) < 1$.

Consider $X\in E$ and suppose that $\det(1-AX)=0 $  and $\|X\|\leq 1$.  We can write $X=\bbm v&w\\0&v \ebm$ where $|w| \leq 1-|v|^2$.  Clearly $|v| \leq 1$.  If $|v|=1$ then $w=0$ and so
\[
0=\det(1-AX)=1-sv+pv^2-a_{21}w= 1-sv+pv^2,
\]
contrary to the assumption that $(s,p)\in\G$.  Hence we have $|v|<1$.  
Moreover
\[
|1-sv+pv^2|= |a_{21}w| \leq|a_{21}|(1-|v|^2)
\]
and so
\[
|a_{21}|\frac{1-|v|^2}{|1-sv+pv^2|} \geq 1,
\]
contrary to the hypothesis \eqref{ass}.  This contradiction shows that 
$X\in E$ and $\det(1-AX)=0$ together imply that $\|X\|>1$.  
A compactness argument shows that the infimum of $\|X\|$ over  $X\in E$ such that $\det(1-AX) =0$ is greater than $1$, or in other words,  $\mu_E(A) < 1$.

The characterization \eqref{mucriterbis} follows by scaling.  Observe that $\mu_E(rA) = r\mu_E(A)$ and so $\mu_E(A) \leq 1$ if and only if $\mu_E(rA) < 1$ for all $r\in(0,1)$.
\end{proof}
\begin{corollary}\label{dependsonly}
For $A\in\mat$ the value of $\mu_E(A)$ depends only on the quantities $\tr A, \det A$ and $a_{21}$.
\end{corollary}
Accordingly we introduce a quotient domain of $\{A: \mu_E(A)<1\}$.
\begin{definition}\label{defCcc}
$\B_\mu$ is the domain in $\mat$ given by 
\beq\label{Bmu}
\B_\mu= \{A\in\mat: \mu_E(A) < 1\}.
\eeq
$\ccc_\mu$ is the domain in $ \C^3$  given by
\beq\label{defDomain}
\PP_\mu= \{(a_{21},\tr A,\det A): A\in\mat, \,  \mu_E(A)<1\} \subset \C^3.
\eeq
\end{definition}
Corollary \ref{dependsonly} asserts that $A\in\mat$ satisfies $A\in\B_\mu$ if and only if $\pi(A)\in\ccc_\mu$.

A major result of the paper is that $\ccc_\mu = \ccc$ (Theorem \ref{4equiv}).

\section{A class of linear fractional functions}\label{Linfrac}

Proposition \ref{mucriter} introduces some linear fractional functions that will play an important role in the paper.
\begin{definition}\label{defPsi}
For $z\in\D$ and $(a,s,p)\in \C^3$ such that $1-sz+pz^2 \neq 0$ let 
\[
\Psi_z(a,s,p)= \frac{a(1-|z|^2)}{1-sz+pz^2}
\]
and let
\[
\kappa(s,p)= \sup_{z\in\D} \frac{1-|z|^2}{|1-sz+pz^2|}.
\]
\end{definition}
Proposition \ref{mucriter} can then be stated:  $\mu_E(A)<1$ if and only if $(\tr A, \det A) \in \G$ and 
\[
\sup_{z\in\D} |\Psi_z (a_{21}, \tr A, \det A)| < 1,
\]
or alternatively, if and only if
\[
|a_{21}| \kappa(\tr A, \det A) < 1.
\]
Recall from Theorem \ref{recapG} that the general point of $\G$ can be written in the form $(\beta+\bar\beta p,p)$ for some $\beta,p\in\D$.
\begin{proposition}\label{valsup}
For $\beta \in\D$ and $(s,p)= (\beta+\bar\beta p,p) \in\G$,
\begin{align}\label{formK}
\kappa(s,p) &= \left|1- \frac{\half s\bar\beta}{1+\sqrt{1-|\beta|^2}}\right|\inv.
\end{align}
Moreover the supremum of $\frac{1-|z|^2}{|1-sz+pz^2|}$ over $z\in\D$ is attained uniquely at the point
\beq\label{attained}
z= \frac{\bar\beta}{1+\sqrt{1-|\beta|^2}}.
\eeq
\end{proposition}
\begin{proof}
Let us first deal with the case that $s=0$.  We have, in terms of $w=1/z^2$,
\[
\kappa(0,p)=\sup_{|w|>1} \frac{|w|-1}{|w+p|}
\]
for $p\in\D$.  Clearly $|w+p| >|w|-1$ when $|w|>1, \, p\in\D$, and so the right hand side is at most $1$.  On letting $w\to \infty$ we see that the supremum is exactly $1$, attained uniquely at $w=\infty$. Thus equation \eqref{formK} is true when $s=0$, attained only at $z=0$, in agreement with equation \eqref{attained} since here $\beta=0$.

Now suppose that $s\neq 0$.  The definition of $\kappa$ can also be written
\[
\kappa(s,p)= \sup_{|z|>1} \frac{|z|^2-1}{|z^2-sz+p|}.
\]
Let
\[
h(z)=z^2-sz+p=u(z)+iv(z)
\]
with $u,v$ real valued and let
\[
g(z) = \frac{|z|^2-1}{|h(z)|}.
\]
We have, at any point other than a zero of $h$,
\begin{align*}
\frac{\partial}{\partial x}|h(z)|&= \frac{\partial}{\partial x} (u^2+v^2)^{\half} = \frac{uu_x+vv_x}{|h(z)|},\\
\frac{\partial}{\partial y}|h(z)|&= \frac{\partial}{\partial y} (u^2+v^2)^{\half} = \frac{vu_x-uv_x}{|h(z)|},\\
\frac{\partial}{\partial x} g(z) &= \frac{\partial}{\partial x}\frac{x^2+y^2-1}{|h(z)|} \\
	&= \frac{|h(z)|2x-(|z|^2-1) \frac{uu_x+vv_x}{|h(z)|}}{|h(z)|^2},\\
\frac{\partial}{\partial y} g(z) &= \frac{|h(z)|2y-(|z|^2-1) \frac{vu_x-uv_x}{|h(z)|}}{|h(z)|^2}.
\end{align*}
At critical points of $g$ in $\{z:|z|>1\}$,
\begin{align*}
(|z|^2-1)(uu_x+vv_x)&=2x|h(z)|^2, \\
(|z|^2-1)(vu_x-uv_x)&=2y|h(z)|^2.
\end{align*}
We may solve these equations to obtain
\[
u_x=\frac{2}{|z|^2-1}(xu+yv), \quad v_x=\frac{2}{|z|^2-1}(xv-yu),
\]
and hence
\beq\label{formh'}
h'(z)=u_x+iv_x= \frac{2}{|z|^2-1}(xh(z)-iyh(z)) = \frac{2\bar z h(z)}{|z|^2-1}.
\eeq

Thus the critical points of $g$ are the points $z$, $|z|>1$, such that
\[
(2z-s)(|z|^2-1)=2\bar z(z^2-sz+p)
\]
or equivalently
\beq\label{eqfors}
s|z|^2-2z-2p\bar z+s=0,
\eeq
whence also 
\[
\bar s |z|^2 - 2\bar pz-2\bar z + \bar s=0.
\]
From these two equations we deduce that
\[
(-2\bar s+2s\bar p)z + (-2\bar s p+2s)\bar z =0
\]
In terms of  $\beta=(s-\bar sp)/(1-|p|^2)$
the last equation becomes $\beta\bar z= \bar\beta z$.  Note that $\beta \neq 0$ since $s\neq 0$. We therefore have $z=r\beta$ for some $r\in\R$.  By virtue of equation \eqref{eqfors}, $r$ must satisfy
\begin{align*}
0&= s|z|^2-2z-2p\bar z+s \\
&= s|z|^2-2 r s +s\\
&= s(|z|^2-2 r +1)\\
&=(\beta+\bar\beta p)(r^2|\beta|^2-2r+1).
\end{align*}
Hence the only possible critical points of $g$ are $z=r\beta$ where
\[
r=\frac{1\pm \sqrt{1-|\beta|^2}}{|\beta|^2}.
\]
It is straightforward to show that $|r\beta|> 1$ only for the plus sign in the above expression, and so we have $z=r\beta$ where
\[
r= \frac{1}{1-\sqrt{1-|\beta|^2}}.
\]
  On retracing our steps we find that $z=r\beta$ is indeed a critical point; thus the nonnegative function $g$ has the unique critical point
\beq\label{critpt}
z=\frac{\beta}{1-\sqrt{1-|\beta|^2}}
\eeq
in $\{z:|z|>1\}$.  By equation \eqref{formh'}, at this point
\begin{align}\label{maxg}
g(z)&=\frac{|z|^2-1}{|h(z)|}\nn \\
	&=\frac{2|z|}{|h'(z)|}\nn  \\
	&=\frac{2|z|}{|2z-s|}\nn \\
	&=\left|1-\frac{s}{2\beta}(1-\sqrt{1-|\beta|^2})\right|\inv  \nn \\
	&=\left| 1-\frac{\half s\bar\beta}{1+\sqrt{1-|\beta|^2}}  \right|\inv.
\end{align}
We claim that $g(z) > 1$.  For any $w\in\C$,
\[
|1-w|<1 \Leftrightarrow \re (1/w) > \half.
\]
Thus
\begin{align*}
g(z)>1 &\Leftrightarrow \re \frac{2\beta}{s(1-\sqrt{1-|\beta|^2})} > \half \\
	&\Leftrightarrow\frac{\beta}{s}+\frac{\bar\beta}{\bar s} > \half (1-\sqrt{1-|\beta|^2})\\
	&\Leftrightarrow\beta(\bar\beta+\beta\bar p)+\bar\beta(\beta+\bar\beta p)> \half|s|^2(1-\sqrt{1-|\beta|^2})\\
	&\Leftrightarrow  4\re(\bar\beta^2 p)+4|\beta|^2 > |\beta+\bar\beta p|^2(1-\sqrt{1-|\beta|^2})\\
	&\Leftrightarrow  4\re(\bar\beta^2 p)+4|\beta|^2 > (|\beta|^2+|\beta p|^2+2\re(\bar\beta^2p))(1-\sqrt{1-|\beta|^2})\\
	& \Leftrightarrow  2(1+\sqrt{1-|\beta|^2})\re (\bar\beta^2p) +(3+\sqrt{1-|\beta|^2})|\beta|^2 > (1-\sqrt{1-|\beta|^2})|\beta p|^2.
\end{align*}
Let $\beta= \omega \cos\theta$ where $\omega\in\T$ and $0<\theta< \half\pi$ (recall that $\beta\neq 0$).  Then
\begin{align*}
g(z)>1 & \Leftrightarrow  2(1+\sin\theta)\cos^2\theta\re(\bar\omega^2p) + (3+\sin\theta)\cos^2\theta >(1-\sin\theta)\cos^2\theta |p|^2 \\
& \Leftrightarrow   3+\sin\theta - (1-\sin\theta)|p|^2+2(1+\sin\theta)\re(\bar\omega^2p) >0.
\end{align*}
Since $\re(\bar\omega^2p) \geq -|p|$, in order to conclude that $g(z)>1$ we need only show that 
\[
3+\sin\theta - (1-\sin\theta)|p|^2 - 2(1+\sin\theta)|p| >0.
\]
But
\begin{align*}
3+\sin\theta - (1-\sin\theta)|p|^2 - 2(1+\sin\theta)|p|&=3-2|p|-|p|^2+(1-2|p|+|p|^2)\sin\theta \\
	&=(1-|p|)(3+|p|+(1-|p|)\sin\theta)\\
	&>0.
\end{align*}
Hence $g(z)>1$ as claimed.   Since $g =0$ on $\T$ and $g(z)\to 1$ as $z\to\infty$, it follows that the unique critical point $z=r\beta$ of $g$ in $\{z:|z|>1\}$ is a global maximum for $g$, and so the maximum $\kappa(s,p)$ of $g$ on $\{z:|z|>1\}$ is indeed given by the value \eqref{formK}, as required.  Moreover, on rewriting the critical point given by equation \eqref{critpt} in terms of the original variable $z\in\D$, we find that the maximum of  $\frac{1-|z|^2}{|1-sz+pz^2|}$ over $z\in\D$ is attained uniquely at
\[
z=\frac{1-\sqrt{1-|\beta|^2}}{\beta}=\frac{\bar\beta}{1+\sqrt{1-|\beta|^2}}.
\]
\end{proof}
On combining Propositions \ref{mucriter} and  \ref{valsup}  we obtain the following description.
\begin{proposition}\label{propmuE}
For any matrix $A=\bbm a_{ij} \ebm \in \mat$,
\[
\mu_E(A) < 1 \quad \mbox{ if and only if } \quad (s,p)\in\G \mbox{ and }|a_{21}| < \left|1-\frac{\half s\bar\beta}{1+\sqrt{1-|\beta|^2}}\right|
\]
where $s=\tr A, \,  p=\det A$ and $\beta=(s-\bar sp)/(1-|p|^2)$.
\end{proposition}
\begin{corollary}\label{altC}
The domain $\PP_\mu$ of Definition {\rm \ref{defCcc}}  satisfies
\beq\label{descrCcc}
\PP_\mu= \left\{(a,s,p): (s,p)\in\G \mbox{ and } |a| < \left| 1-\frac{\half s\bar\beta}{1+\sqrt{1-|\beta|^2}}\right| \right\}
\eeq
where $\beta=(s-\bar sp)/(1-|p|^2)$.
\end{corollary}

\section{The domains $\mathcal P$ and $\ccc_\mu$}\label{twodomains}

The purpose of this section is to show that $\PP=\PP_\mu$ and to give criteria for membership of the domain.  One inclusion is easy.
\begin{proposition}\label{CandC1}
$\PP\subset\PP_\mu.$
\end{proposition}
\begin{proof}
Consider $(a,s,p)\in \PP$ and pick $A=\bbm a_{ij} \ebm \in\mat$ such that $\|A\|<1, \ \ a_{21}=a, \ \tr A=s,  \ \det A=p$.  Since $\mu_E \leq \|\cdot\|$ for all subspaces $E$ of $\mat$ we have $\mu_E(A) < 1$, and hence, by Definition \ref{defCcc}, $(a,s,p)\in\PP_\mu$.
\end{proof}

The next result provides characterizations of points in $\ccc$ and asserts that $\ccc=\ccc_\mu$.
\begin{theorem}\label{4equiv}
Let
\beq\label{las2}
(s,p)=(\beta+\bar\beta p,p) = (\la_1+\la_2,\la_1\la_2)\in\G
\eeq
and let $a\in\C$.
The following statements are equivalent.
\begin{enumerate}
\item $(a,s,p)\in\PP$;
\item $(a,s,p) \in \PP_\mu$;
\item $|a|<\left|1- \frac{\half s\bar\beta}{1+\sqrt{1-|\beta|^2}}\right|$;
\item $|a| < \half|1-\bar\la_2\la_1|+\half (1-|\la_1|^2)^\half(1-|\la_2|^2)^\half$;
\item $\sup_{z\in\D} \left|\Psi_z(a,s,p)\right|  < 1$.
\end{enumerate}
\end{theorem}
\begin{proof}
We shall show that (1) $\Rightarrow$ (2) $\Rightarrow$ (5) $\Rightarrow$ (3) $\Rightarrow$ (4) $\Rightarrow$ (1).
Indeed, (1) $\Rightarrow$ (2) is Proposition \ref{CandC1} while (2) $\Leftrightarrow$ (5) is Proposition \ref{mucriter}.

\noindent (5) $\Rightarrow$ (3)  If (5) holds then (see Definition \ref{defPsi}) $|a|\kappa(s,p) <1$ and hence, by Proposition \ref{valsup}, (3) holds.

\noindent (3) $\Rightarrow$ (4)  We shall show that the right hand sides in (3) and (4) are equal, that is,
\beq\label{toshow}
\half|1-\bar\la_2\la_1|+\half \La =
\left|1- \frac{\half s\bar\beta}{1+\sqrt{1-|\beta|^2}}\right|
\eeq
where
\[
\Lambda=(1-|\la_1|^2)^\half(1-|\la_2|^2)^\half.
\]
Let $L,\ R$ denote the left and right hand sides respectively of equation \eqref{toshow} and  let
\[
L_1=L(1+\sqrt{1-|\beta|^2})(1-|\la_1\la_2|^2), \qquad R_1=R(1+\sqrt{1-|\beta|^2})(1-|\la_1\la_2|^2).
\]
Since
\[
\beta=\frac{s-\bar sp}{1-|p|^2}=\frac{\la_1(1-|\la_2|^2)+\la_2(1-|\la_1|^2)}{1-|\la_1\la_2|^2},
\]
we find that
\begin{align}\label{beta}
1-|\beta|^2 &= \frac{(1-|\la_1\la_2|^2)^2 - | \la_1(1-|\la_2|^2)+\la_2(1-|\la_1|^2)|^2}{(1-|\la_1\la_2|^2)^2}\nn \\
&= \frac{1-2|\la_1\la_2|^2 + |\la_1\la_2|^4} {(1-|\la_1\la_2|^2)^2} \nn \\
&\hspace*{0.5cm}+\frac{-\{| \la_1|^2 (1-|\la_2|^2)^2+|\la_2|^2(1-|\la_1|^2)^2 + 2(1-|\la_2|^2)(1-|\la_1|^2) \re(\bar\la_2\la_1) \} }{(1-|\la_1\la_2|^2)^2} \nn \\
&= \frac{(1-|\la_2|^2)(1-|\la_1|^2)}{(1-|\la_1\la_2|^2)^2} \{ 1 + |\la_1\la_2|^2 - 2 \re(\bar\la_2\la_1)\} \nn \\
&= \frac{|1-\bar\la_2\la_1|^2 \La^2}{(1-|\la_1\la_2|^2)^2} 
\end{align}
Thus 
\beq\label{p-beta-la}
\sqrt{1-|\beta|^2}=\frac{|1-\bar\la_2\la_1|\La}{1-|\la_1\la_2|^2}.
\eeq
Hence 
\begin{align}\label{Lex}
L_1&= \half(|1-\bar\la_2\la_1|+\La)(1-|\la_1\la_2|^2+|1-\bar\la_2\la_1|\La) \nn \\
	&=\half|1-\bar\la_2\la_1|\left(1-|\la_1\la_2|^2+(1-|\la_1|^2)(1-|\la_2|^2)\right) \nn \\
	&\hspace*{2cm}+\half\La\left(|1-\bar\la_2\la_1|^2+1-|\la_1\la_2|^2\right) \nn \\
	&=\half|1-\bar\la_2\la_1|(2-|\la_1|^2-|\la_2|^2) + \La(1-\re(\bar\la_2\la_1))
\end{align}
Now let $\zeta$ be a square root of $1-\bar\la_2\la_1$: we find that equation \eqref{Lex} may be written
\beq\label{Lex2}
L_1=L(1+\sqrt{1-|\beta|^2})(1-|\la_1\la_2|^2)= \half\left|\zeta(1-|\la_1|^2)^\half+\bar\zeta(1-|\la_2|^2)^\half\right|^2.
\eeq
Next we express $R_1$ in terms of $\la_1$ and $\la_2$.  Observe that
\begin{align*}
s(\bar s - s \bar p) &=(\la_1+\la_2)(\bar\la_1\left(1-|\la_2|^2)+\bar\la_2(1-|\la_1|^2)\right) \\
	&=|\la_1|^2+|\la_2|^2-2|\la_1\la_2|^2+(1-|\la_1|^2)(1-\zeta^2)+(1-|\la_2|^2)(1-\bar\zeta^2)\\
	&=2-2|\la_1\la_2|^2 -(1-|\la_1|^2)\zeta^2-(1-|\la_2|^2)\bar\zeta^2.
\end{align*}
Thus
\begin{align}\label{Rex}
R_1 &=(1-|\la_1\la_2|^2)\left|1+\sqrt{1-|\beta|^2}-\half s\bar\beta\right| \nn \\
	&=\left|1-|\la_1\la_2|^2 +|1-\bar\la_2\la_1|\La -\half s(\bar s- s\bar p)\right| \nn \\
	&=\left|1-|\la_1\la_2|^2 +|1-\bar\la_2\la_1|\La -\half   (2-2|\la_1\la_2|^2 -(1-|\la_1|^2)\zeta^2-(1-|\la_2|^2)\bar\zeta^2)    \right| \nn \\
	&=\half \left| 2|\zeta|^2\La +(1-|\la_1|^2)\zeta^2+(1-|\la_2|^2)\bar\zeta^2\right| \nn \\
	&=\half \left|\zeta(1-|\la_1|^2)^\half + \bar\zeta(1-|\la_2|^2)^\half\right|^2 \nn \\
	&=L_1. \nn
\end{align}
Hence $L=R$ and so (3) $\Leftrightarrow$ (4).

\noindent (4) $\Rightarrow$ (1) is Proposition \ref{asuffCondn}.  Hence all five conditions are equivalent.
\end{proof}

There is an analogue of Theorem \ref{4equiv} for the closures of $\ccc$ and $\ccc_\mu$.  Note that by \cite[Theorem 1.1]{AYmodel}, $(s,p)\in\Gamma$ if and only if $|p|\leq 1$ and there exists $\beta\in\C$ such that $|\beta|\leq 1$ and $s=\beta+\bar\beta p$.  In the case that  $(s,p)\in\Gamma$ and $|p|=1$ then $s=\beta+\bar\beta p$ where $\beta=\half s$.
Indeed,  $(s,p)=(\la_1+\la_2,\la_1\la_2)\in\Gamma$ and $\la_1,\la_2\in\T$.  Hence $s=\bar s p$.    Let $\beta=\half s$.  Then $\beta+\bar\beta p= \half s+ \half \bar s p =s$. (Infinitely many other choices of $\beta$ are also possible when $|p|=1$.)

Observe also that if $(s,p)\in\Gamma$ and  $z\in\D$ then $1-sz+pz^2\neq 0$.
\begin{theorem}\label{4equivClosure}
Let
\beq\label{las2bis}
(s,p)=(\beta+\bar\beta p,p) = (\la_1+\la_2,\la_1\la_2)\in\Gamma
\eeq
where $|\beta| \leq 1$  and if $|p|=1$ then $\beta=\half s$.  Let $a\in\C$.
The following statements are equivalent.
\begin{enumerate}
\item $(a,s,p)\in\bar\PP$;
\item $(a,s,p) \in \bar\PP_\mu$;
\item $|a|\leq\left|1- \frac{\half s\bar\beta}{1+\sqrt{1-|\beta|^2}}\right|$;
\item $|a| \leq \half|1-\bar\la_2\la_1|+\half (1-|\la_1|^2)^\half(1-|\la_2|^2)^\half$;
\item $ \left|\Psi_z(a,s,p)\right|  \leq 1$ for all $z\in\D$; \blue
\item there exists $A\in\mat$ such that $\|A\|\leq 1$ and $\pi(A)=(a,s,p)$;
\item there exists $A\in\mat$ such that $\mu_E(A)\leq 1$ and $\pi(A)=(a,s,p)$. \black
\end{enumerate}
\end{theorem}
\begin{proof}
\blue  (1) $\Rightarrow$ (6) Suppose (1).  Pick a sequence $x_n \in\ccc$ such that $x_n \to (a,s,p)$ and then, for every $n$, pick $A_n\in\B$ such that $\pi(A_n)=x_n$.  Pass to a convergent subsequence of $(A_n)$, with limit $A\in\bar\B$.  Then
\[
\pi(A)= \lim \pi(A_n)= \lim x_n = (a,s,p).
\]
Thus (6) holds.

(6) $\Rightarrow$ (7) is immediate from the fact that $\mu_E(A) \leq \|A\|$ for all $A\in\mat$.

(7) $\Rightarrow$ (1)  Let $A$ be as in (7). For any $r\in (0,1)$ we have $\mu_E(rA) < 1$ and $\pi(rA)=(ra,rs,r^2p)$.  By Theorem \ref{4equiv} $(ra,rs,r^2p) \in \ccc$.  Let $r\to 1$ to conclude that $(a,s,p) \in\bar\ccc$.
\black

Having proved (1), (6) and (7) equivalent we again show that (1)  $\Rightarrow$ (2)
$\Rightarrow$ (5) $\Rightarrow$ (3) $\Rightarrow$ (4) $\Rightarrow$ (1).
As above, (1) $\Rightarrow$ (2) is immediate from Proposition \ref{CandC1} while (2) $\Leftrightarrow$ (5) follows from Proposition \ref{mucriter}.

(5) $\Rightarrow$ (3)  If (5) holds then $|a|\kappa(s,p) \leq 1$ and so, by Proposition \ref{valsup}, (3) holds.

(3) $\Rightarrow$ (4)  Suppose (3).  If $|p|<1$ then the right hand sides in conditions (3) and (4) are equal by the argument in the proof of Theorem \ref{4equiv}.  Suppose therefore that $|p|=1$.  By hypothesis $\beta=\half s$ and
\begin{align*}
|a| &\leq \left| 1- \frac{\quart|s|^2}{1+\sqrt{1-\quart|s|^2}} \right|\\
	&=\sqrt{1-\quart|s|^2}.
\end{align*}
The right hand side of (4) is
\begin{align*}
\half|1-\bar\la_2\la_1|=\half|\la_1-\la_2|=\half |s^2-4p|^\half=|\quart s(s\bar p )-1|^\half=\sqrt{1-\quart|s|^2}.
\end{align*}
Once again the right hand sides in (3) and (4) are equal, and so (3) $\Leftrightarrow$ (4).

(4) $\Rightarrow$  (1) is contained in Proposition \ref{anothersuff}.
\end{proof}

\section{Elementary geometry of the pentablock}\label{elemGeom} 

In this section we give some basic geometric properties of the pentablock $\PP$ and its closure. 

\begin{theorem} \label{not-convex} Neither
$\PP$ nor $\bar\PP$ is convex.
\end{theorem}
\begin{proof}
If $x=(0,2,1)= (0,1+1,1\cdot 1)$ and $y=(0,2i, -1)=(0,i+i,i \cdot i)$ then $x,y \in \bar\PP$, but the mid-point of these two points is
$\tfrac 12 (x+y)= (0,1+i,0) \notin \bar\PP$. Thus $\PP$ is not convex.
\end{proof}

However, $\bar\PP$ is contractible by virtue of the following result.
\begin{theorem} \label{thm-starlike}
$\PP$ and $\bar\PP$ are $(1,1,2)$-quasi-balanced and are starlike about $(0,0,0)$, but not  circled. 
\end{theorem}
The statement that $\ccc$ is $(1,1,2)$-quasi-balanced means that if $(a,s,p)\in\ccc$ and $z\in\Delta$ then 
$(za,zs,z^2p)\in\ccc$.
\begin{proof}
The quasi-balanced property follows from the fact that, for $A\in\mat$ and $z\in\C$, if $\pi(A)=(a,s,p)$ then
$\pi(zA)=(za,zs,z^2p)$.

Let $x = (a, s,p) \in \PP$ and write $(s,p)= (\la_1+\la_2,\la_1\la_2)\in\G$. By Theorem \ref{4equiv},
 $x\in\PP$ if and only if
\beq\label{a-in-P}
|a| < \half|1-\bar\la_2\la_1|+\half (1-|\la_1|^2)^\half(1-|\la_2|^2)^\half.
\eeq
Let $0 <r <1$ and let $(rs,rp)= (\gamma_1+\gamma_2,\gamma_1\gamma_2),$
so that $\gamma_1, \gamma_2$ are the roots of
\[
\gamma^2 -r s \gamma + r p =0.
\]
To show that $\PP$ is starlike about $(0,0,0)$ we need to show that
\[
|ra| < \half|1-\bar\gamma_2\gamma_1|+\half (1-|\gamma_1|^2)^\half(1-|\gamma_2|^2)^\half.
\]

Suppose it is not true, that is, there exists a choice of $r, a$ such that \eqref{a-in-P} holds, but
$$|a| \ge \frac{1}{2r}\{|1-\bar\gamma_2\gamma_1|+(1-|\gamma_1|^2)^\half(1-|\gamma_2|^2)^\half\}.$$
Thus we have
\beq\label{not-starlike}
\frac{1}{2r}\{|1-\bar\gamma_2\gamma_1|+ (1-|\gamma_1|^2)^\half(1-|\gamma_2|^2)^\half \} <
\half|1-\bar\la_2\la_1|+\half (1-|\la_1|^2)^\half(1-|\la_2|^2)^\half.
\eeq 
To show that $\PP$ is starlike about $(0,0,0)$ we must prove that the inequality \eqref{not-starlike} never happens for any $\la_1, \la_2 \in \D$ and $r \in (0,1)$, that is,
\beq\label{starlike}
|1-\bar\gamma_2\gamma_1|+(1-|\gamma_1|^2)^\half(1-|\gamma_2|^2)^\half \ge r \{|1-\bar\la_2\la_1|+ (1-|\la_1|^2)^\half(1-|\la_2|^2)^\half \}
\eeq 
holds for all $\la_1, \la_2 \in \D$ and $r \in (0,1)$.

The inequality \eqref{starlike} is equivalent to
\begin{align}\label{starlike-1}
|1-\bar\gamma_2\gamma_1|^2+ (1-|\gamma_1|^2)(1-|\gamma_2|^2) +
2|1-\bar\gamma_2\gamma_1|(1-|\gamma_1|^2)^\half(1-|\gamma_2|^2)^\half  \ge & ~\nn\\
\;\;\;\;r^2 \{|1-\bar\la_2\la_1|^2 +(1-|\la_1|^2)(1-|\la_2|^2)+
2|1-\bar\la_2\la_1|(1-|\la_1|^2)^\half(1-|\la_2|^2)^\half\} & ~.
\end{align}

By equation \eqref{keepthis},
\begin{align}
1-\half|s|^2+|p|^2 &=\half (1-|\la_1|^2)(1-|\la_2|^2) +\half |1-\bar\la_2\la_1|^2. 
\end{align}
Thus \eqref{starlike} is equivalent to
\begin{align}\label{starlike-2}
2-r^2|s|^2+2 r^2|p|^2 +
2|1-\bar\gamma_2\gamma_1|(1-|\gamma_1|^2)^\half(1-|\gamma_2|^2)^\half  \ge & ~\nn\\
\;\;\;\;r^2 \{2-|s|^2+2|p|^2+
2|1-\bar\la_2\la_1|(1-|\la_1|^2)^\half(1-|\la_2|^2)^\half\}, & ~
\end{align}
and therefore to
\begin{align}\label{starlike-3}
2(1-r^2)+
2|1-\bar\gamma_2\gamma_1|(1-|\gamma_1|^2)^\half(1-|\gamma_2|^2)^\half  \ge & ~\nn\\
\;\;\;\;2 r^2 |1-\bar\la_2\la_1|(1-|\la_1|^2)^\half(1-|\la_2|^2)^\half. & ~
\end{align}
By equation \eqref{p-beta-la},
\beq\label{p-beta-la-1}
\sqrt{1-|\beta|^2}(1-|p|^2)=|1-\bar\la_2\la_1|(1-|\la_1|^2)^\half(1-|\la_2|^2)^\half,
\eeq
where 
\[
\beta=\frac{s-\bar sp}{1-|p|^2}.
\]
Hence \eqref{starlike} is equivalent to
\beq\label{starlike-4}
1+ \sqrt{1-|\beta_r|^2}(1-r^2|p|^2)\ge r^2\{1 +\sqrt{1-|\beta|^2}(1-|p|^2)\},
\eeq
where 
\[
\beta_r=\frac{rs-r^2\bar sp}{1-r^2|p|^2}.
\]
Therefore to show that $\PP$ is starlike about $(0,0,0)$ it is enough to show that the function $f : (0,1) \to \R$,
\[
f(r) = \frac{1}{r^2}\{1+ \sqrt{1-|\beta_r|^2}(1-r^2|p|^2)\}
\]
is monotone decreasing on $(0,1)$.
Let us prove that the derivative $f'(r) < 0$ for all $r \in (0,1)$. 

A straightforward verification shows that, for any  $ r > 0$,
\begin{align}\label{der_f_1}
f'(r) &= \frac{-2}{r^3} \{ 1+ \sqrt{1-|\beta_r|^2}(1-r^2|p|^2) \}
+\frac{1}{r^2} \left(\sqrt{1-|\beta_r|^2}(1-r^2|p|^2)\right)'\\
~&=-\frac{2}{r^3} - \frac{2}{r^3} \sqrt{1-|\beta_r|^2}(1-r^2|p|^2)+ \;\;\;\;\;\;\;\;\;\;\;\;\;\;\;\;\;\;\;\;\;\;\nn \\
~&~ \frac{1}{r^2} \{-2 r |p|^2 \sqrt{1-|\beta_r|^2} +
(1-r^2|p|^2) \frac{(-1)}{2\sqrt{1-|\beta_r|^2}} 
(\beta_r \bar{\beta_r})'\}.\nn 
\end{align}
Thus
\begin{equation}\label{der_f_2}
f'(r)= -\frac{2}{r^3} - \frac{2}{r^3}\sqrt{1-|\beta_r|^2}-\frac{1}{r^2}(1-r^2|p|^2) \frac{1}{2 \sqrt{1-|\beta_r|^2}} (\beta_r \bar{\beta_r})'.
\end{equation}

Another straightforward calculation shows that, for any  $ r > 0$,
\[
(\beta_r)'= \left(\frac{rs-r^2\bar sp}{1-r^2|p|^2}\right)' =\frac{s -r \bar{s} p- p r(\bar{s}-r s \bar{p})}{(1-r^2|p|^2)^2}.
\]
Hence
\begin{align}\label{der-beta_r-bar-beta_r}
(\beta_r \bar{\beta_r})'&= \beta_r' \bar{\beta_r} + \beta_r \bar{\beta_r'}= 2 \mathrm{Re}(\beta_r' \bar{\beta_r}) \\
&= 2 \mathrm{Re}\left(\bar{\beta_r} \frac{s -r \bar{s} p- p r(\bar{s}-r s \bar{p})}{(1-r^2|p|^2)^2}\right)  \nn\\
&=  \frac{2}{(1-r^2|p|^2)}\mathrm{Re}
\left\{
\bar{\beta_r} \left(\frac{rs -r^2 \bar{s} p}{r(1-r^2|p|^2)}-\frac{ p (r\bar{s}-r^2 s \bar{p})}{(1-r^2|p|^2)} \right)
\right\}  \nn\\
&=  \frac{2}{(1-r^2|p|^2)}\mathrm{Re}
\left\{
\bar{\beta_r} \left(\frac{1}{r}\beta_r - p \bar{\beta_r} \right)
\right\}.  \nn
\end{align}

Therefore, by \eqref{der_f_2} and \eqref{der-beta_r-bar-beta_r},
we have 
\begin{align}\label{der_f_3}
f'(r) &= -\frac{2}{r^3} - \frac{2}{r^3}\sqrt{1-|\beta_r|^2}-\\
&\;\;\;\; \frac{1}{r^2}(1-r^2|p|^2) \frac{1}{2 \sqrt{1-|\beta_r|^2}}
 \frac{2}{(1-r^2|p|^2)}\mathrm{Re}
\left\{
\bar{\beta_r} \left(\frac{1}{r}\beta_r - p \bar{\beta_r} \right)
\right\}  \nn\\
 &= -\frac{2}{r^3} - \frac{2}{r^3}\sqrt{1-|\beta_r|^2}-\frac{1}{r^2} \frac{1}{\sqrt{1-|\beta_r|^2}}
\mathrm{Re}
 \left(\frac{1}{r}|\beta_r|^2 - p \bar{\beta_r}^2 \right)  \nn\\
 &= -\frac{2}{r^3} - \frac{1}{r^3}\frac{(2-|\beta_r|^2)}{\sqrt{1-|\beta_r|^2}} + \frac{1}{r^2} \frac{1}{\sqrt{1-|\beta_r|^2}}
\mathrm{Re}(p \bar{\beta_r}^2).  \nn
\end{align}

By \cite[Theorem 2.3]{AY04}, $\G$ is starlike about $(0,0)$. Hence
$(s,p) \in \G$ implies that $(rs,rp) \in \G$ for all $0 <r <1$, and,
by \cite[Theorem 2.1]{AY04}, we have $|\beta_r| <1$. Therefore
\[
-1 < \mathrm{Re}(p \bar{\beta_r}^2) < 1.
\]
Hence, for all $r \in (0,1)$,
\beq \label{der-f-final}
-\frac{2}{r^3} - \frac{1}{r^3}
\frac{(2-|\beta_r|^2)}{\sqrt{1-|\beta_r|^2}} - 
\frac{1}{r^2} \frac{1}{\sqrt{1-|\beta_r|^2}}
 < f'(r) < -\frac{2}{r^3} - \frac{1}{r^3}
\frac{(2-|\beta_r|^2)}{\sqrt{1-|\beta_r|^2}} + 
\frac{1}{r^2} \frac{1}{\sqrt{1-|\beta_r|^2}}.
\eeq
The right-hand side of \eqref{der-f-final} can be expressed as
\begin{align}\label{RHS-der_f}
{\rm RHS} &= -\frac{2}{r^3} - \frac{1}{r^3}\frac{(2-|\beta_r|^2)}{\sqrt{1-|\beta_r|^2}} + \frac{1}{r^2} \frac{1}{\sqrt{1-|\beta_r|^2}}\\
~&= -\frac{1}{r^3} \left(2 +\frac{(2-|\beta_r|^2)}{\sqrt{1-|\beta_r|^2}} -  \frac{r}{\sqrt{1-|\beta_r|^2}}\right) \nn \\
~&= -\frac{1}{r^3} \left( 2 +\sqrt{1-|\beta_r|^2} +
\frac{1-r}{\sqrt{1-|\beta_r|^2}}\right).  \nn
\end{align}
Thus  $f'(r) <0$ for all $r \in (0,1)$. This implies that 
$\PP$ is starlike about $(0,0,0)$.

The point $x=(0,2,1)$ is in $\bar\PP$, but $ix = (0,2i,i) \notin\bar\PP$ because, for $(0,2i,i)$,
\[ 
|s -\bar{s} p| =|2i+2i \cdot i|= |2i - 2| > 0\; \text{but} \;
1-|p|^2=0.
\]
Therefore neither $\bar\PP$ nor $\PP$ is circled.
\end{proof}

A domain $\Omega$ is said to be polynomially convex provided
that, for each compact subset $K$ of $\Omega$, the polynomial hull $\widehat K$ of $K$ is contained in $\Omega$.

\begin{theorem} \label{polyconvex}
$\PP$ and $\bar\PP$ are polynomially convex.
\end{theorem}
\begin{proof} Let us first show that $\bar\PP$ is polynomially convex.
Let $x \in \C^3\setminus\bar\PP$.  We must find a polynomial $f$ such that $|f|\leq 1$ on $\bar\PP$ and $|f(x)| > 1$.   

If $(x_2,x_3) \notin\Gamma$ then, since $\Gamma$ is polynomially convex  \cite[Theorem 2.3]{AY04}, there is a polynomial $g$ in two variables such that $|g|\leq 1$ on $\Gamma$ and $|g(x_2,x_3)|>1$.  The polynomial $f(u_1,u_2,u_3)= g(u_2,u_3)$ then separates $x$ from $\bar\ccc$.

Now suppose that $(x_2,x_3)=(\la_1+\la_2,\la_1\la_2) \in\Gamma$.  By  Theorem \ref{4equivClosure} it must be that
\[
|x_1| >  \half|1-\bar{\la_2}\la_1|+\half (1-|\la_1|^2)^\half(1-|\la_2|^2)^\half.
\]

If $|x_1| >1$ the polynomial $f(u) = u_1$ has the desired property.
Otherwise   $|x_1| \le 1$.   Recall that, for all $(a,s,p)\in \bar\PP$,
\[
\left|\Psi_z(a,s,p)\right|=\left|\frac{a(1-|z|^2)}{1-sz+pz^2}\right| \le 1
\]
for all $z \in \D$. By Proposition \ref{valsup},  the point
$$
z_0 = \frac{\bar\beta}{1 + \sqrt{1 - |\beta|^2}}\in\D,
$$ 
where $ \beta=\frac{s-\bar sp}{1-|p|^2}$,
satisfies 
$|\Psi_{z_0}(x)|> 1$, while $|\Psi_{z_0}| \leq 1$ on $\bar\PP$. We shall approximate the linear
fractional function $\Psi_{z_0}$ by a polynomial.  For $N\geq 1$ let
\[
g_N(a,u_1,u_2)= a (1-|z_0|^2)
(1+ z_0 u_1 +\dots+z_0^N u_1^N)(1+ z_0 u_2 +\dots+z_0^N u_2^N).
\]
Then $g_N$ is a polynomial that is symmetric in $u_1$ and $u_2$.  Hence there is a polynomial
  $f_N$ in $3$ variables such that
\[
f_N(a, u_1 +u_2, u_1 u_2)=g_N(a,u_1,u_2).
\]
For any complex $z,w$ different from $1$ we have
\begin{align*}
(1-z)\inv(1-w)\inv-\sum_0^N z^j\sum_0^N w^k &=\sum_0^N z^j\frac{w^{N+1}}{1-w}+ \frac{z^{N+1}}{(1-z)(1-w)}
\end{align*}
and hence if $|z|<1, \, |w|<1$,
\begin{align*}
\left|(1-z)\inv(1-w)\inv-\sum_0^N z^j\sum_0^N w^k\right| &\leq \frac{|z|^{N+1}+|w|^{N+1}}{(1-|z|)(1-|w|)}.
\end{align*}
For any $ u_1,u_2$ such that $   |u_1| \le 1, \ |u_2|\leq 1$ substitute $z=u_1z_0, \ w=u_2z_0$ and deduce that
\begin{align*}
\left|(1-z_0u_1)\inv(1-z_0u_2)\inv-\sum_0^N z_0^j u_1^j  \sum_0^N z_0^k u_2^k \right|&\leq \frac{2|z_0|^{N+1}}{(1-|z_0|)^2}.
\end{align*}
It follows that if $|a|\leq 1, |u_1|\leq 1, |u_2|\leq 1$ then
\begin{align*}
|(f_N-\Psi_{z_0})(a,u_1+u_2,u_1u_2)|&= |g_N(a,u_1,u_2) - \Psi_{z_0}(a,u_1+u_2,u_1u_2)|\\
	&  \le  |a|(1-|z_0|^2) \frac{2|z_0|^{N+1}}{ (1 - |z_0|)^2}\\
	& \leq \frac{4|a||z_0|^{N+1}}{ 1 - |z_0|}.
\end{align*}

Let $0< \eps< \tfrac 13 (|\Psi_{z_0}(x)|-1)$ and choose $N$ so large that $|f_N -\Psi_{z_0}| < \eps$  at all points  $ (a, u_1 +u_2, u_1 u_2)$ such that $|a|\leq 1, |u_1| \le 1, |u_2|\leq 1$.   
Then $|f_N| < 1+\eps$ on $\bar\PP$ and $|f_N(x)| \geq 1+2\eps$.
The function $f=(1+\eps)^{-1}f_N$ has the desired properties.
 Thus $\bar\PP$ is polynomially  convex.

Now consider any compact subset $K$ of $\PP$.  For $r\in
(0, 1)$ define the compact set
$$
\PP_r \stackrel{\mathrm{def}}{=}
\{(z_0,z_1 + z_2,
z_1z_2):  |z_1| \le r, |z_2| \le r, |z_0|\le \half|1-\bar{z_2}z_1|+\half (1-|z_1|^2)^\half(1-|z_2|^2)^\half\}.
$$
Then
\[
\bigcup_{0<r<1} \ccc_r = \ccc,
\]
and so, for $r$ sufficiently close to $1$, we have
$$ 
K \subset \PP_r \subset \PP.
$$
Since $\PP_r$ is polynomially convex,
$$
\widehat K \subset \widehat\PP_r = \PP_r
\subset \PP,
$$
and so $\PP$ is polynomially convex. 
\end{proof}

It follows that $\PP$ is a domain of holomorphy (for example \cite[Theorem 3.4.2]{krantz}).  However, Theorem \ref{PcapR3} shows that $\PP$ does not have a $C^1$ boundary, and consequently much of the theory of pseudoconvex domains does not apply to $\PP$.

\section{Some automorphisms of $\PP$}\label{automorphisms}

By an {\em automorphism} of a domain $\Omega$ in $\C^n$ we mean a holomorphic map $f$ from $\Omega$ to $\Omega$ with holomorphic inverse.  Every bijective holomorphic self-map of $\Omega$ is in fact an automorphism \cite{krantz}.

For $\alpha \in \C$ we write
$$
B_\alpha(z) = \frac{z-\alpha}{1-\overline \alpha z}.
$$
In the event that $\al\in\D$ the rational function $B_\al$ is called
a {\em Blaschke factor}.
A {\em M\"obius function} is a function of the form $cB_\alpha$
for some $\alpha \in \mathbb{D}$ and $c\in \mathbb{T}$.
The set of all M\"obius functions is the automorphism group 
$\Aut \mathbb{D}$ of $\mathbb{D}$.

All automorphisms of the symmetrised bidisc $\G$ are induced by elements of  $ \Aut \D$ \cite{JP}. That is, they are of the form
\[
\tau_\up (z_1 +z_2, z_1 z_2) = (\up(z_1) +\up(z_2), \up(z_1) \up(z_2)), \;\; z_1, z_2 \in \D,
\]
for some $\up \in \Aut \D$.
See also \cite[Theorem 4.1]{AY08} for another proof of this result.

For $\omega \in \T$ and $\up \in \Aut \D$, let
\beq \label{automP}
f_{\omega \up}( a, s,p)= 
\left( \frac{\omega \eta(1-|\alpha|^2) a}{1-  \bar\alpha s + \bar{\alpha}^2 p}, \tau_{\up}(s,p) \right)
\eeq
where $\up = \eta B_{\alpha}$.

\begin{theorem}\label{aut_PP} 
The maps $f_{\omega \up}$, for 
$\omega \in \T$ and $\up \in \Aut \D$, constitute a group of automorphisms of $\PP$ under composition. Each automorphism $f_{\omega \up}$ extends analytically to a neighbourhood  of $\bar\PP$.

Moreover, for all $\omega_1, \omega_2 \in \T$, $\up_1, \up_2 \in \Aut \D$,
\[
f_{\omega_1 \up_1} \circ f_{\omega_2 \up_2} = f_{(\omega_1 \omega_2) (\up_1 \circ \up_2)},
\] 
and, for all $\omega \in \T$, $\up \in \Aut \D$,
\[
(f_{\omega \up})^{-1} = f_{\bar{\omega} \up^{-1}}.
\]
\end{theorem}

One can use Theorem \ref{4equiv} and straightforward calculations to prove these statements. In this paper we will take a different approach. We show in  Proposition \ref {prop7.2} to Corollary \ref{7.5} below  that this group is the image under a homomorphism induced by $\pi$ of a group of automorphisms of $\B$.  Moreover the explicit formula \eqref{automP-la}
shows that every rational function $f_{\omega\ups}$ extends holomorphically to a neighbourhood of $\bar\ccc$.

For $\omega \in \T$ and $\up \in \Aut \D$ we define 
\[
F_{\omega \up}: \B \to \B
\] 
by 
\beq \label{automBall}
F_{\omega \up}(A)= \up(U_{\omega}A U_{\omega}^*),\;\; A \in \B,
\eeq
where
 \begin{equation*}
 U_{\omega}= \left[ \begin{array}{cc} 1 & 0 \\ 0 & \omega
\end{array} \right].
\end{equation*}
Note that $\up(U_{\omega} A U_{\omega}^*)$ is well defined by the functional calculus since the spectrum $\sigma (U_{\omega}A U_{\omega}^*)$ is contained in $\D$. If $\up = \eta B_{\alpha}$ then
\[
\up(A) = \eta B_{\alpha} (A) = \eta (A - \alpha I) (I -\bar{\alpha} A)^{-1}.
\]
It is easy to see that
\[
F_{\omega \up}(A)= U_{\omega} \up(A) U_{\omega}^*.
\]

\begin{proposition} \label{prop7.2}
The set 
\[
\mathcal{F} = \{ F_{\omega \up} : \omega \in \T, \; \up \in \Aut \D\}
\]
is a group of automorphisms of $\B$ under composition, and
\[
F_{\omega_1 \up_1} \circ F_{\omega_2 \up_2} = F_{(\omega_1 \omega_2) (\up_1 \circ \up_2)}
\]
and 
\[
(F_{\omega \up})^{-1} = F_{\bar\omega \up^{-1}}.
\]
\end{proposition}
\begin{proof}
For $\omega_1, \omega_2 \in \T$, $\up_1, \up_2 \in \Aut \D$ and for all $A \in \B$,
\begin{align} \label{compF}
\left(F_{\omega_1 \up_1} \circ F_{\omega_2 \up_2}\right)(A)&= F_{\omega_1 \up_1} (\up_2(U_{\omega_2} A U_{\omega_2}^*))\\
&=  \up_1( U_{\omega_1}  \up_2(U_{\omega_2}  A U_{\omega_2}^*) U_{\omega_1}^*) \nn\\
&=  \up_1(\up_2(U_{\omega_1}  U_{\omega_2}  A U_{\omega_2}^* U_{\omega_1}^*))  \nn\\
&=  F_{(\omega_1 \omega_2) (\up_1 \circ \up_2)}(A).\nn
\end{align}
For $\omega \in \T$, $\up \in \Aut \D$, 
\begin{align} \label{inverseF}
F_{\omega \up} \circ F_{\bar\omega \up^{-1}}&= F_{(\omega \bar\omega) (\up \circ \up^{-1})}\\
& = F_{ (1) (\id_{\D})} = \id_{\B}. \nn
\end{align}
\end{proof}

\begin{proposition} If $A_1, A_2 \in \B$ and $\pi(A_1) = \pi(A_2)$
then, for any  $\omega \in \T$ and $\up \in \Aut \D$,
\[
\pi (F_{\omega \up}(A_1)) = \pi(F_{\omega \up}(A_2)).
\]
Furthermore, if $\pi(A_1) = (a, s, p)$ then 
\[
\pi (F_{\omega \up}(A_1)) = \left( \frac{\omega \eta(1-|\alpha|^2) a}{1-  \bar\alpha s + \bar{\alpha}^2 p}, \tau_{\up}(s,p) \right)
\]
where $\up = \eta B_{\alpha}$ for  $\eta \in \T$ and $\alpha \in  \D$.
\end{proposition}
\begin{proof}
Let $A = (a_{ij})_{i,j=1}^2 \in \B$; then 
\begin{align} \label{piF}
\pi (F_{\omega \up}(A))&= \pi (U_{\omega}\up(A) U_{\omega}^*)\\
&= \pi (U_{\omega} \eta (A - \alpha I) (I -\bar{\alpha} A)^{-1} U_{\omega}^*). \nn
\end{align}
Straightforward calculations show that
\[
(I -\bar{\alpha} A)^{-1} = \frac{1}{1- \bar{\alpha} \tr (A) + \bar{\alpha}^2 \det (A)} \left[ \begin{array}{cc} 1- \bar{\alpha} a_{22} & \bar{\alpha} a_{12} \\ \bar{\alpha} a_{21} & 1- \bar{\alpha}a_{11}
\end{array} \right].
\]
Thus
\begin{align} \label{upA}
\up(A) &= \frac{\eta}{1- \bar{\alpha} \tr (A) + \bar{\alpha}^2 \det (A)}
\left[ \begin{array}{cc} a_{11} - \alpha & a_{12} \\ a_{21} & a_{22} -\alpha
\end{array} \right]
 \left[ \begin{array}{cc} 1- \bar{\alpha} a_{22} & \bar{\alpha} a_{12} \\ \bar{\alpha} a_{21} & 1- \bar{\alpha}a_{11}
\end{array} \right] 
\end{align}
and
\begin{align} \label{UupAU*}
U_{\omega} \up(A)U_{\omega}^*  &= \frac{\eta}{1- \bar{\alpha} \tr (A) + \bar{\alpha}^2 \det (A)}
\left[ \begin{array}{cc} * & * \\ \omega a_{21}(1 - |\alpha|^2) & *
\end{array} \right]
\end{align}
By the spectral mapping theorem, if $\sigma(A)= \{ \la_1, \la_2 \}$ then
\begin{align} \label{sigma(UupAU*)}
\sigma(F_{\omega \up}(A))&= \sigma(U_{\omega} \up(A)U_{\omega}^*)\\
&= \sigma(\up(A)) = \{ \up(\la_1), \up(\la_2 )\}.
\end{align}
Therefore if $\pi(A) = (a, s, p)$ then
\[
(\tr, \det)(F_{\omega \up}(A)) = \tau_\up(s,p)
\]
and
\[
\pi (F_{\omega \up}(A))= \left( \frac{\omega \eta(1-|\alpha|^2) a}{1-  \bar\alpha s + \bar{\alpha}^2 p}, \tau_{\up}(s,p) \right).
\]
\end{proof}
\begin{corollary} Each automorphism $F_{\omega \up} \in \mathcal{F}$
induces an automorphism $f_{\omega \up}$ of $\PP$ by 
\[
f_{\omega \up}(a,s,p)= \pi (F_{\omega \up}(A))
\]
for any $A \in \B$ such that $\pi(A) = (a, s, p)$.
Moreover, the map
\[
\chi: \mathcal{F} \to \Aut \PP\; \text{ defined by} \;\;\; \chi (F_{\omega \up}) =  f_{\omega \up}
\]
is a homomorphism of groups.
\end{corollary}
\begin{proof} Let $\omega_1, \omega_2 \in \T, \; \up_1, \up_2 \in \Aut \D$. Consider $(a, s, p) \in \PP$ and pick $A \in \B$ such that $\pi(A) = (a, s, p)$. Then
\begin{align} \label{homomorphism}
(f_{\omega_1 \up_1} \circ f_{\omega_2 \up_2})( a, s, p)&= 
f_{\omega_1 \up_1} (\pi (F_{\omega_2 \up_2}(A)))\\
&=  \pi (F_{\omega_1 \up_1}(F_{\omega_2 \up_2}(A))) \nn \\
&=  \pi (F_{\omega_1 \up_1} \circ F_{\omega_2 \up_2}(A)) \nn\\
&= \chi (F_{\omega_1 \up_1} \circ F_{\omega_2 \up_2})( a, s, p). \nn
\end{align}
Thus $ \chi (F_{\omega_1 \up_1} \circ F_{\omega_2 \up_2})=f_{\omega_1 \up_1} \circ f_{\omega_2 \up_2}$ for all $\omega_1, \omega_2 \in \T, \; \up_1, \up_2 \in \Aut \D$.
\end{proof}
\begin{corollary} \label{7.5}
The set 
\[
\chi({\mathcal{F}}) = \{ f_{\omega \up}:  \omega \in \T, \up \in \Aut \D \}
\]
is a group of automorphisms of $\PP$ under composition.
\end{corollary}

\begin{proposition}\label{7.6}
For $\omega \in \T$, $\up \in \Aut \D$, and for all $(s,p) \in \PP$, 
\begin{align} \label{automP-la}
f_{\omega \up}( a, s, p)&= 
 \frac{\eta}{1-  \bar\alpha s + \bar{\alpha}^2 p}
\left( \omega (1-|\alpha|^2) a,\; -2 \alpha+ (1+|\alpha|^2) s -2 \bar{\alpha} p,\; \eta (\alpha^2 -\alpha s +p) \right),
\end{align}
where $\up = \eta B_\alpha$ for  $\eta \in \T$ and $\alpha \in  \D$.
\end{proposition}

Since the appearance of the first version of this paper at arXiv:1403.1960, L. Kosinski \cite{kos} has shown that $\chi(\mathcal{F})$ is in fact
the full group of automorphisms of $\PP$.

\section{The distinguished boundary of $\PP$}\label{disting_bound}

Let $\Omega$ be a domain in $\C^n$ with closure $\bar\Omega$ and let $A(\Omega)$ be the algebra of continuous scalar functions on $\bar\Omega$ that are holomorphic on $\Omega$.  A {\em boundary} for $\Omega$ is a subset $C$ of $\bar\Omega$ such that every function in $ A(\Omega)$ attains its maximum modulus on $C$.  It follows from the  
theory of uniform algebras  \cite[Corollary 2.2.10]{browder} that (at least when $\bar\Omega$ is polynomially convex, as in the case of $\PP$) there is a smallest closed boundary of $\Omega$, contained in all the closed boundaries of $\Omega$ and called the {\em distinguished boundary} of $\Omega$ (or the {\em Shilov boundary} of $A(\Omega)$).  In this section we shall determine the distinguished boundary of $\PP$; 
we denote it by $b\PP$.

Clearly, if there is a function $g\in A(\PP)$  and a point $u \in \bar\PP$ such that $g(u) = 1$ and $ | g(x) | < 1$ for all $x \in \bar\PP \setminus\{u\}$, then $u$ must belong to $b\PP$. Such a point $u$ is called a {\em peak point} of $\bar\PP$ and the function $g$ a {\em peaking function} for $u$.

By \cite[Theorem 2.4]{AY04}, the distinguished boundary of  $\Gamma$ is the
{\em symmetrized torus}
$$
b \Gamma = \{(z_1+z_2, z_1z_2): z_1, z_2 \in
\mathbb{T}\}
$$
which is homeomorphic to a M\" obius band. 
\begin{proposition}\label{peak_points_G}
Every point of $b \Gamma$ is a peak point of $\Gamma$.
\end{proposition}
\begin{proof}
Consider $(s,p)=(z_1+z_2, z_1 z_2)$ where $z_1, z_2 \in \T$. 
If $z_1=z_2$ then the function $f(\zeta_1, \zeta_2) = \tfrac 14 (\zeta_1 + s)$ peaks at $(s,p)$.
If $z_1 \neq z_2$, let $\phi$ be a conformal map of $\D$ onto the open elliptic region ${\mathcal E}$ with major axis $(-1,1)$ and minor axis of length less than $2$. By Carath\'eodory's theorem, $\phi$ extends continuously to map $\Delta$ bijectively onto $\bar{\mathcal E}$.
We can suppose (replacing $\phi$ by its composition with a Blaschke factor) that $\phi(z_1)=1$ and $\phi(z_2)=-1$. The function
\[
\tilde{g}(\zeta_1, \zeta_2) = \tfrac 14 (\phi(\zeta_1) -\phi(\zeta_2))^2
\]
is a symmetric function in $A(\D^2)$ that attains its maximum modulus on $\Delta^2$ only at the points $(z_1, z_2)$ and $(z_2, z_1)$, and hence induces a function $g \in A(\Gamma)$ that peaks at $(s,p)$.
\end{proof} 

Define 
$$ K_0 \df \{ (a,s,p) \in \C^3: (s,p) \in b \Gamma, |a|= \sqrt{1 -\tfrac 14 |s|^2} \}$$
and
$$
 K_1  \df \{ (a,s,p) \in \C^3: (s,p) \in b \Gamma, |a|\le \sqrt{1 -\tfrac 14 |s|^2} \}.
$$
The set of $2 \times 2$ unitary matrices is denoted by ${\mathcal{U}}(2)$.

\begin{proposition} \label{unitary_matrix_bPP}
$ \pi({\mathcal{U}}(2)) = K_1$.  
\end{proposition}
\begin{proof}
 By Theorem \ref{4equivClosure}, $\pi({\mathcal{U}}(2)) \subset \bar\PP$ and $|a| \le \half |1-\bar\la_2 \la_1|= \sqrt{1- \tfrac 14 |s|^2}.$ Thus $\pi({\mathcal{U}}(2)) \subset K_1$.

Suppose $(a,s, p) \in  K_1$. To prove that $\pi({\mathcal{U}}(2)) = K_1$ we need to find  a $2 \times 2$ unitary matrix $U$ 
such that $(a,s, p)=\pi(U)$. Since $(s,p) \in 
b \Gamma$ there exist $  \la_1, \la_2 \in
\mathbb{T}$ such that $s = \la_1+\la_2$ and $p= \la_1 \la_2$.
Let 
 \begin{equation*}
 U = V^* \left[ \begin{array}{cc} \la_1 &  0 \\ 0 & \la_2
\end{array} \right] V,
\end{equation*}
where, for some $\eta \in \mathbb{T}$ and $\theta \in \R$,
 \begin{equation*}
 V = \left[ \begin{array}{cc} \cos \theta & \eta \sin \theta  \\ 
-\sin \theta  & \eta \cos \theta
\end{array} \right].
\end{equation*}
Thus
 \begin{equation*}
 U = \left[ \begin{array}{cc} \la_1 \cos^2 \theta  + \la_2 \sin^2 \theta  & (\la_1\eta -\la_2\eta)  \sin \theta \cos \theta \\ 
(\la_1 \bar\eta -\la_2 \bar\eta)  \sin \theta \cos \theta  & \la_1 \sin^2 \theta  + \la_2 \cos^2 \theta
\end{array} \right]
\end{equation*}
is a unitary matrix.
Let $w = \half (\la_1 - \la_2)$. For $(a,s, p) \in  K_1$, we have
$|a| \le |w|$. We need to find $\eta \in \mathbb{T}$ and $\theta \in \R$
such that $a =  \bar\eta w \sin (2\theta)$.

If $w=0$, then $a=0$, and one can take 
 \begin{equation*}
 U = \left[ \begin{array}{cc} \la_1 &  0 \\ 0 & \la_2
\end{array} \right].
\end{equation*}
If $w \neq 0$, then $|\frac{a}{w}| \le 1$. We can choose $\eta \in \mathbb{T}$ such that $\frac{a}{w} \eta \in \R $, and choose
 $\theta \in \R$
such that $\sin (2\theta)= \frac{a}{w}\eta $.
Then $(a,s, p)=\pi(U)$.  Hence $\pi({\mathcal{U}}(2)) = K_1$.
\end{proof}
We shall use the notation $D(a; r)$ to mean the open disc centred at $a\in\C$ with radius $r>0$.
\begin{proposition} \label{closed_boundariesbPP}
The subsets $K_0$ and $K_1$ of $\bar\PP$ are closed boundaries for $A(\PP)$.
\end{proposition}
\begin{proof}
 To show that $K_1$ is a closed boundary for $A(\PP)$
consider any $f \in A(\PP)$. Then $ f \circ \pi \in A(\mathbb{B})$, where $\mathbb{B}$ is the $2 \times 2$ matrix ball. Since  ${\mathcal{U}}(2)$ is the distinguished boundary of $\mathbb{B}$ \cite[Section 4.6]{clerc}, there exists $U \in {\mathcal{U}}(2)$ such that $ f \circ \pi$ attains its maximum modulus at $U$. Hence $f$ attains its  maximum modulus at $\pi(U)$. 
Therefore $\pi({\mathcal{U}}(2))$ is a closed boundary for $A(\PP)$.
By Proposition \ref{unitary_matrix_bPP},
$\pi({\mathcal{U}}(2)) = K_1$.

 Let us show that $K_0$ is a closed boundary for $A(\PP)$.
Consider $f \in A(\PP)$. Since $K_1$ is a closed boundary for $A(\PP)$,
there exists $(s,p) \in  b \Gamma$ such that $f$ attains its  maximum modulus on the disc 
\[
D(0;\sqrt{1 -\tfrac 14 |s|^2}) \times  \{(s,p) \} \subset \partial \PP,
\]
say at the point $(a,s,p)$. Then $f$ must also attain its  maximum modulus at a point $(a_0,s,p)$ for some $a_0$ such that $|a_0|=\sqrt{1 -\tfrac 14 |s|^2}$.
Otherwise
\[
|f(a,s,p)| > \sup_{|z| = \sqrt{1 -\tfrac 14 |s|^2}} | f(z,s,p)|.
\] 
It follows that, for some $r \in (0,1)$ sufficiently close to $1$,
\[
|f(ra,rs,rp)| > \sup_{|\theta| = r \sqrt{1 -\tfrac 14 |s|^2}} | f(\theta,r s,r p)|.
\] 
Since $f$ is analytic in a neighbourhood of the disc
\[
r D(0;\sqrt{1 -\tfrac 14 |s|^2}) \times  \{(rs,rp) \},
\]
which is a subset of $\PP$ by the starlike property of $\PP$, this contradicts the maximum principle applied to $ f(\cdot, rs, rp)$.

Thus $f$ attains its  maximum modulus at a point  of $K_0$. Hence $K_0$ is a closed boundary for $A(\PP)$.
\end{proof}
\begin{theorem} \label{characbPP}
For $x \in \C^3$, the following are equivalent.
\begin{enumerate}
\item[(1)] $ x \in K_0$;
\item[(2)] $x $ is a peak point of $\bar\PP$;
\item[(3)]  $x \in b\PP$, the distinguished boundary of $\PP$.
\end{enumerate}
Therefore 
\[
b\PP = \{ (a,s,p)\in \C^3: (s,p) \in b\Gamma, |a|= \sqrt{1 - \tfrac 14 |s|^2} \}
\]
and so 
\[
b\PP = \{(a,s,p) \in \C^3: |s| \leq 2, |p| = 1, s = \bar s p
 \mbox{  and  } |a|= \sqrt{1 - \tfrac 14 |s|^2} \}.
\]
\end{theorem}
\begin{proof} 
\noindent (1)$\Rightarrow$(2)
We will exhibit a peaking function for an arbitrary point $(a,s,p)\in K_0$.  

Since $(s,p)\in b\Gamma$ there exist $\la_1,\la_2\in\T$ such that $s=\la_1+\la_2, \, p=\la_1\la_2$.
Consider first the case that $\la_1=\la_2$.  Then $|s|=2$ and so $|a|^2=1-\quart|s|^2=0$.  Thus
$(a,s,p)=(0,2\la_1, \la_1^2)$.   Let $f(x)=2\la_1+x_2$.  Clearly $|f|\leq 4$ on $\bar\ccc$, attained for $x\in\bar\ccc$ such that $x_2=2\la_1$.  The only such $x\in\bar\ccc$ is $x=(0,2\la_1, \la_1^2)$, and so $f$ is a peaking function for $(a,s,p)$.

Now suppose that $\la_1\neq\la_2$.  Choose an automorphism $\upsilon$ of $\D$ such that $\up(\la_1)=1$ and $\up(\la_2)=-1$.   The automorphism $\tau_\up$ of $\G$ induced by $\up$ (or more precisely, the continuous extension of $\tau_\up$ to $\Gamma$) maps $(s,p)$ to $(0,-1)$.    By Theorem \ref{aut_PP}, $\up$ induces an automorphism $\kappa$ of $\PP$ which extends analytically to a neighbourhood of $\bar\ccc$ and is bijective on $\bar\ccc$. This $\kappa$ maps $(a,s,p)$ to a point $(b,0,-1)$ for which $|b|=1$.   Consider the function $f(x) = (b + x_1) g (x_2, x_3)$ where
$ g \in A(\Gamma)$ peaks at $(0,-1)$ and $g(0,-1) =1$. Then $\|f \|_\infty =2$ and $|f(b,0,-1) | =2$, and if $|f(x)| =2$ for some $x \in \bar\PP$ then $|b + x_1| =2$ and 
$| g (x_2, x_3)| =1$. Hence $x_1 =b$ and $(x_2,x_3) = (0,-1)$, that is,  $f$
peaks at $(b,0,-1)$ and consequently $f \circ \kappa$ is  a peaking function for $\kappa^{-1} (b,0,-1)= (a,s,p)$.
Thus {\rm (1)}$\Rightarrow${\rm (2)}.

\noindent (2)$\Rightarrow$(3) holds since peak points always belong to the distinguished boundary.

\noindent  {\rm (3)}$\Rightarrow${\rm (1)} is 
Proposition \ref{closed_boundariesbPP}.

Thus (1), (2) and (3) are equivalent.

 By Theorem \ref{characbPP},
\[
b\PP = \{ (a,s,p)\in \C^3: (s,p) \in b\Gamma, |a|= \sqrt{1 - \tfrac 14 |s|^2} \}.
\]
As in \cite{AY04} an element
$(s,p) \in \C^2$ lies in $b\Gamma$ if and only if
\begin{equation}
\label{rband}
|s| \leq 2 \mbox{  and  } |p| = 1 \mbox{  and  } s = \bar s p.
\end{equation}
\end{proof}

\begin{theorem}\label{top-dist-boundary} The distinguished boundary $b\PP$ is homeomorphic to
$$
\{(\sqrt{1-x^2}\omega, x, \theta): -1 \le x \le 1,\quad 0 \le \theta \le 2 \pi, \quad \omega\in\T\}
$$
with the two points $(\sqrt{1-x^2} \omega, x, 0)$ and $(\sqrt{1-x^2} \omega,-x, 2\pi)$ identified for every $\omega\in\T$ and $x\in [-1,1]$.
\end{theorem}
\begin{proof} We have
\[
 b\PP = \{ (a,s,p)\in \C^3: (s,p) \in b\Gamma, |a|= \sqrt{1 - \tfrac 14 |s|^2} \}
\]
\[
= \{ (a,z_1+z_2, z_1z_2)\in \C^3: z_1, z_2 \in
\mathbb{T} \; \text{and} \; |a|= \sqrt{1 - \tfrac 14 |z_1+z_2|^2} \}.
\]
Let us
write $z_1z_2 = e^{i\theta}$: then
$$z_1+z_2 = z_1+\overline z_1 e^{i\theta} = e^{i\theta/2}\ 2\
\mathrm{Re}(z_1 e^{-i\theta/2}),$$
and we may parametrize $b\PP$ by
$$
b\PP = \{(\sqrt{1-x^2}e^{i\eta}, 2xe^{i\theta/2}, e^{i\theta}): -1 \le x \le 1,\quad 0 \le \theta \le 2 \pi, \quad 0 \le \eta \le 2 \pi\}.
$$
Thus $b\PP$ is homeomorphic to the set
$$
\{(\sqrt{1-x^2} e^{i\eta}, x, \theta): -1 \le x \le 1,\quad 0 \le \theta \le 2 \pi, \quad 0 \le \eta \le 2 \pi\}
$$
with the points $(\sqrt{1-x^2} e^{i \eta}, x, 0)$ and $(\sqrt{1-x^2} e^{i\eta},-x, 2\pi)$ identified for every $\eta: 0 \le \eta \le 2 \pi$.
\end{proof}

\section{The real pentablock $\ccc\cap\R^3$} \label{real_pentablock}

  We shall show that the real pentablock is a convex body bounded by five faces, comprising two triangles, an ellipse and two curved surfaces.

It will be helpful if we first recall the shape of the real symmetrised bidisc.

\begin{proposition}\label{pr-realgamma}
$\Gamma \cap \R^2$ is the isosceles triangle with vertices $(\pm 2,1)$ and $(0,-1)$ together with its interior.
\end{proposition}
\begin{proof} 
By Theorem \ref{recapG} 
if $s$ and $p$ are real, then
\begin{eqnarray}\label{GammaReal}
(s, p) \in \G &\Leftrightarrow& |s(1-p)| < 1 -p^2 \nn \\
&\Leftrightarrow& |p|  < 1\ \  \mbox{and}\ \ |s| < 1+p.
\end{eqnarray}
Thus the plane $\mathrm{Im}\ s = \mathrm{Im}\ p =0$ intersects $\G$ in the
interior of the isosceles triangle with vertices at $(0, -1)$
and $(\pm 2, 1)$.  
\end{proof}
\begin{figure}\label{realgamma}
\includegraphics{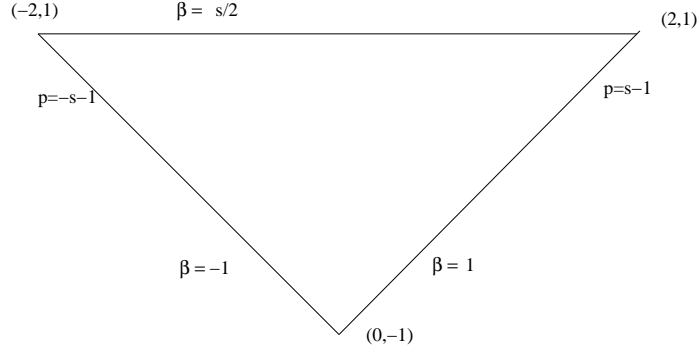}
\caption{The real symmetrised bidisc}
\end{figure}
Figure 1 indicates the values of the parameter $\beta$, where $s=\beta+\bar\beta p$, on the sides of the triangle.  At the vertex $(0,-1)$, one can take $\beta$ to be any real number.

Although $\PP$ is not convex, $\PP\cap \R^3$ {\em is}.

\begin{theorem} The  real pentablock
$\PP\cap \R^3$ is convex.
\end{theorem}
\begin{proof} 
Let $(a_1,s_1,p_1), (a_2,s_2,p_2) \in \PP\cap \R^3$. By Theorem \ref{4equiv}, $(s_1,p_1), (s_2,p_2) \in \G\cap \R^2$,
 $|a_1| < K(s_1,p_1)$ and $|a_2| < K(s_2,p_2)$, where for  $(s,p) \in \G$
\[
K(s,p) = \left|1- \frac{\half s\bar\beta}{1+\sqrt{1-|\beta|^2}}\right|
\]
and $\beta=\frac{s-\bar sp}{1-|p|^2}$.

By Proposition \ref{pr-realgamma}  $\G\cap \R^2$ is convex. To prove that $\PP\cap \R^3$ is convex we have  to show that 
for all $0 <t <1$,
\[
|t a_1 + (1-t) a_2| < K (t(s_1,p_1) + (1-t)(s_2,p_2)).
\]
Note that 
\[
|t a_1 + (1-t) a_2| \le t|a_1| + (1-t) |a_2|< t K (s_1,p_1) + (1-t) K(s_2,p_2).
\]
Thus it suffices to prove that 
for all $0 <t <1$,
\[
t K(s_1,p_1) + (1-t) K(s_2,p_2) \le K (t(s_1,p_1) + (1-t)(s_2,p_2)),
\]
that is, that $K:\G \cap \R^2 \to \R$ is concave.

For real $(s,p) \in \G$, $\beta=\frac{s}{1+p}$ and
$-1 <\beta <1$.  Thus
\[
K(s,p) = \left|1- \frac{\half s\beta}{1+\sqrt{1-\beta^2}}\right|
\]
\[
=1- \frac{\half s\beta (1-\sqrt{1-\beta^2})}{1-(1-\beta^2)}
=1- \half \frac{1}{\beta} s(1-\sqrt{1-\beta^2})
\]
\[
=1- \half (1+p) (1-\sqrt{1-\beta^2})=1- \half \left(1+p -\sqrt{(1+p)^2 -s^2}\right).
\]
It is straightforward to show that the Hessian of $K$ 
\begin{equation*}
\left[ \begin{array}{cc} 
\frac{\partial^2 K}{\partial s^2} & 
\frac{\partial^2 K}{\partial s \partial p} \\
\\
\frac{\partial^2 K}{\partial s \partial p}&
 \frac{\partial^2 K}{\partial p^2}
\end{array} \right]
= \frac{1}{2((1+p)^2 -s^2)^{3/2}}
\left[ \begin{array}{cc} 
-(1+p)^2 & s(1+p) \\
\\
s(1+p) & -s^2
\end{array} \right] \le 0.
\end{equation*}
Therefore $K$ is concave and $\PP\cap \R^3$ is convex.
\end{proof}
\begin{figure}
  \label{realpenta}
  \includegraphics[width=7.5cm]{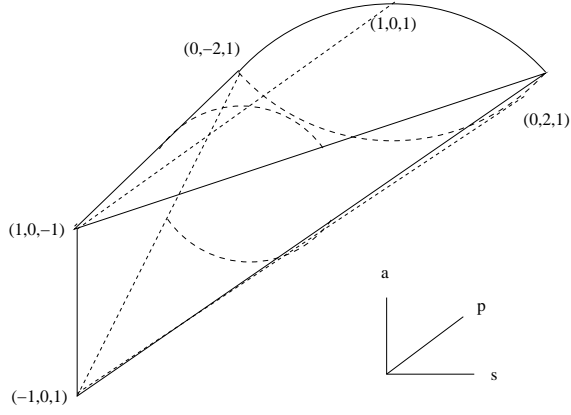}
 \caption{The real pentablock}
\end{figure}
The sketch of the real pentablock in Figure 2 is explained in the following statement.
\begin{theorem} \label{PcapR3}
$\PP\cap \R^3$ is a convex open domain with five faces and with the four vertices $(0,-2,1), (0,2,1)$, $ (1,0,-1)$ and $(-1,0,-1)$.
The faces are the following sets:
\begin{enumerate}
\item the triangle with vertices $ (0,2,1)$, $ (1,0,-1)$ and $(-1,0,-1)$  together with its interior;
\item the triangle with vertices $(0,-2,1)$, $ (1,0,-1)$ and $(-1,0,-1)$  together with its interior;
\item the ellipse
$$\{(a,s,1): a ^2+ s^2/4 = 1, -2 \le s \le 2 \}$$
with  centre at $(0,0,1)$, with major axis joining the points  $(0,2,1)$ and $(0,-2,1)$ and with minor axis joining the points $(1,0,1)$ and $ (-1,0,1)$,  together with its interior;
\item a surface with vertices  $ (1,0,-1)$ and $(0,-2,1)$, $(0,2,1)$
and boundaries \\
(i) $\{(a,s,1): a = \sqrt{1- s^2/4}, -2 \le s \le 2 \}$; \\
(ii) the straight line segment joining  $(0,-2,1)$ and $ (1,0,-1)$; \\
(iii) the straight line segment joining $ (0,2,1)$ and $ (1,0,-1)$;
\item a surface with vertices  $ (-1,0,-1)$ and $(0,-2,1)$, $(0,2,1)$
and boundaries \\
(i) $\{(a,s,1): a = -\sqrt{1- s^2/4}, -2 \le s \le 2 \}$; \\
(ii)  the straight line segment joining $(0,-2,1)$ and $ (-1,0,-1)$; \\
(iii) the straight line segment joining $ (0,2,1)$ and $ (-1,0,-1)$.
\end{enumerate} 
\end{theorem}
\begin{proof} By Corollary \ref{altC},
the domain $\PP$ is expressible by the equation
\beq\label{descrP-1}
\PP= \left\{(c,s,p): (s,p)\in \G, |c| < \left| 1-\frac{\half s\bar\beta}{1+\sqrt{1-|\beta|^2}}\right| \right\}
\eeq
where $\beta=(s-\bar sp)/(1-|p|^2)$. 
By \eqref{GammaReal}, $(s,p) \in \Gamma \cap \R^2$ if and only if  $s \in \R$ and $|s(1-p)| \le 1-p^2$, that is,
$s \in \R$, $ -1 \le p \le 1$ and $|s| \le 1+ p$. For $(s, p) \in \R^2$, $\beta=s(1-p)/(1-p^2)= s/(1+p)$.

Therefore, 
\begin{align}\label{descrP-2}
\PP \cap \R^3 &= \left\{
(a,s,p): (s,p)\in\G \cap \R^2,  a \in \R\; \text{and} \; |a| < \left| 1-\frac{\half s^2/(1+p)}{1+\sqrt{1-(s/(1+p))^2}} \right| 
\right\}. \nn
\end{align}

Let us consider the boundary of $\PP\cap \R^3$.
\begin{enumerate}
\item Let $\beta =1$,  and so $s=1+p$, $|a| \le  |1 - \half s|$.
Thus we have
a triangle with vertices: $ (0,2,1)$, $ (1,0,-1)$ and $(-1,0,-1)$;
\item Let $\beta =-1$,  and so $-s=1+p$, $|a| \le  |1 + \half s|$.
Thus we have a triangle which has vertices: $(0,-2,1)$, $ (1,0,-1)$ and $(-1,0,-1)$;
\item Let $p=-1$, then $s =0$. Thus we have a straight line between 
two points $(-1,0,-1)$ and $(1,0,-1)$.

Let $p=1$ and so $\beta =\half s$. Then
\[
|a| \le \left| 1-
\frac{\half s \beta}{1+\sqrt{1-\beta^2}} \right| = \sqrt{1-(\half s)^2}.
\]
Therefore we have the ellipse 
$$\{(a,s,1): a ^2+ s^2/4 \le 1, -2 \le s \le 2 \}$$
with centre at $(0,-0,1)$ which goes through the points $(1,0,1)$, $(0,2,1)$, $ (-1,0,1)$ and $(0,-2,1)$;
\item the surface $S_1$ is
\[
 \left\{(a,s,p): (s,p)\in\G \cap \R^2,  a \in \R, 0\le a \le 1\; \text{and} \; a = \left| 1-\frac{\half s^2/(1+p)}{1+\sqrt{1-(s/(1+p))^2}} \right| 
\right\}
\]
which has vertices  $ (1,0,-1)$ and $(0,-2,1)$, $(0,2,1)$
and boundaries\\
(i) $\{(a,s,1): a = \sqrt{1- s^2/4}, -2 \le s \le 2 \}$; \\
(ii) the straight segment joining  $(0,-2,1)$ and $ (1,0,-1)$; \\
(iii)  the straight segment joining $ (0,2,1)$ and $ (1,0,-1)$;
\item the surface $S_2$ is
\[
 \left\{(a,s,p): (s,p)\in\G \cap \R^2,  a \in \R, -1 \le a \le 0\; \text{and} \; a = - \left| 1-\frac{\half s^2/(1+p)}{1+\sqrt{1-(s/(1+p))^2}} \right| 
\right\}
\]
which has  vertices  $ (-1,0,-1)$, \, $(0,-2,1)$ and $(0,2,1)$
and boundaries: 
(i) $\{(a,s,1): a = -\sqrt{1- s^2/4}, -2 \le s \le 2 \}$; \\
(ii)  the straight line segment joining  $(0,-2,1)$ and $ (-1,0,-1)$; \\
(iii) the straight line segment joining  $ (0,2,1)$ and $ (-1,0,-1)$.
\end{enumerate} 
\end{proof}

\section{A Schwarz Lemma for a general $\mu$ }\label{schwarz_mu}
The classical Schwarz Lemma gives a solvability criterion for a two-point interpolation problem in $\D$.  There is a simple  analogue for two-point $\mu$-synthesis; it is general in terms the cost functions $\mu_E$ to which it applies, but very special in terms of the interpolation conditions.   In this section we consider a general linear subspace $E$ of $\C^{n\times m}$ and the corresponding $\mu_E$ on $\C^{m\times n}$, as in equation \eqref{defmu}.

\begin{definition}\label{defB_mu}
$\Omega_{\mu_E}$ is the domain in $\C^{m\times n}$ given by 
\beq\label{Bmu_n}
\Omega_{\mu_E}= \{A\in \C^{m\times n}: \mu_E(A) < 1\}.
\eeq
\end{definition}

We shall denote by $N$ the Nevanlinna class of functions on the disc \cite{rudin} and if $F$ is a matricial function on $\D$ then we write $F\in N$ to mean that each entry of $F$ belongs to $N$.  It then follows from Fatou's Theorem that if $F\in N$ is an $m\times n$-matrix-valued function then
\[
\lim_{r\to 1-} F(r\la) \mbox{ exists for  almost all }  \la \in\T.
\]

\begin{lemma}\label{mu_E}
Let $F,G \in \hol(\D,\C^{m\times n})$ satisfy $F(\la) = \la G(\la)$ for all $\la \in \D$. Let $F\in N$ and let $E$ be a subset of $\C^{n\times m}$. Suppose that  $ \mu_E(F(\la)) < 1 $ for all $\la \in \D$.
Then $ \mu_E(G(\la)) \le 1 $ for all $\la \in \D$.
\end{lemma}
\begin{proof}
Write
\[
F_*(\la) = \lim_{r \to 1-} F(r \la) 
\]
for $\la \in \T$ where the limit exists.
Clearly
\begin{align}\label{mu_E_n}
\mu_E(F_*(\la)) & \le 1  \;\; \text{ exists for almost all}\;\; \la \in \T, \\
\mu_E(\la G_*(\la)) & \le 1  \;\; \text{ exists for almost all}\;\; \la \in \T, \nn\\
\mu_E(G_*(\la)) & \le |\la| \mu_E(\la G_*(\la)) \le 1 \;\; \text{  for almost all}\;\; \la \in \T. \nn
\end{align}
By the maximum principle for $\mu_E$ \cite[Theorem 8.21]{DuPa},
$ \mu_E(G(\la)) \le 1 $ for all $\la \in \D$.
\end{proof}
\begin{proposition}\label{schwarz_L_B_mu}
Let $\la_0 \in \D \setminus \{0 \}$, let $W \in \C^{m\times n}$ and let $E$ be a subset of $\C^{n\times m}$. There exists $F \in  N\cap\hol(\D,\C^{m\times n})$
such that 
\begin{enumerate}
\item $F(0) =0$ and $F(\la_0)= W$,
\item $\mu_E(F(\la)) < 1 $ for all $\la \in \D$
\end{enumerate}
if and only if $\mu_E(W) \le |\la_0|$. 
\end{proposition}
\begin{proof} ($\Leftarrow$) Suppose $\mu_E(W) \le |\la_0|$. Let 
$F(\la) = \frac{\la}{\la_0} W$. Then $F\in N, \, F(0) =0, \, F(\la_0)= W$ and, 
for all $\la \in \D$,
\[
\mu_E(F(\la)) = \mu_E\left(\frac{\la}{\la_0} W\right) =\frac{|\la|}{|\la_0|}\mu_E(W) \le |\la|< 1.
\]
($\Rightarrow$) Suppose  there exists $F\in N$ such that (1) and (2) hold. Since  $F(0) =0$ there exists  $G \in \hol(\D,\C^{m\times n})$ such that $F(\la) = \la G(\la)$ for all $\la \in \D$ and 
\[ 
G(\la_0) = \frac{1}{\la_0} F(\la_0) = \frac{1}{\la_0} W.
\]
By Lemma \ref{mu_E}, $\mu_E(G(\la_0)) \le 1$. Hence $\mu_E(W) \le |\la_0|$.
\end{proof}

In the next section we shall seek a Schwarz Lemma for $\ccc$.  One might try to deduce such a result from Proposition \ref {schwarz_L_B_mu}  by lifting maps from $\hol(\D,\ccc)$ to $\hol(\D,\Omega_{\mu_E})$.  However,  
 Section \ref{lifting} shows that the lifting problem is delicate, and a Schwarz Lemma for $\PP$ cannot easily be derived in this way.

\section{What is the Schwarz Lemma for $\PP$?}\label{schwarz}

For which pairs $\la_0\in\D$ and $(a,s,p) \in\PP$ does there exist $h\in\hol(\D,\PP)$ such that $h(0)=(0,0,0)$ and $h(\la_0)=(a,s,p)$?   We can easily find a necessary condition.
\begin{proposition}\label{necSchwarz}
If $h\in\hol(\D,\PP)$ satisfies $h(0)=(0,0,0)$ and $h(\la_0)=(a,s,p)$ then
\beq\label{schwarzGamma}
\frac{2|s -\bar sp|+|s^2-4p|}{4-|s|^2} \leq |\la_0|
\eeq
and
\beq\label{2ndnec}
|a|\left/ \left|1-\frac{\half s \bar{\beta}}{1+\sqrt{1-|\beta|^2}}\right|\right. \leq |\la_0|
\eeq
where $\beta= (s-\bar s p)/(1-|p|^2)$.
\end{proposition}
\begin{proof}
If $h=(h_1,h_2,h_3)$ then $(h_2,h_3) \in\hol(\D,\G)$ maps $0$ to $(0,0)$ and $\la_0$ to $(s,p)$.
By the Schwarz Lemma for $\G$ \cite[Theorem 1.1]{AY01} the inequality \eqref{schwarzGamma} holds. 

By Theorem \ref{4equiv}, for every $z\in\D$, the function
\[
\Psi_z(a,s,p)=\frac{a(1-|z|^2)}{1-sz+pz^2}
\]
maps $\PP$ analytically to $\D$. It also maps $(0,0,0)$ to  $0$.  Hence  $\Psi_z\circ h$ is an analytic self-map of $\D$ that maps $0$ to $0$ and $\la_0$ to $\Psi_z(a,s,p)$.  By Schwarz' Lemma we have 
\[
|\Psi_z(a,s,p)| \leq |\la_0| \quad \mbox{ for all } z\in\D.
\]
On taking the supremum of the left hand side over $z\in\D$ and invoking Proposition \ref{valsup} we obtain the inequality \eqref{2ndnec}.
\end{proof}
On dividing through by $\la_0$ in the inequalities \eqref{schwarzGamma} and \eqref{2ndnec} and letting $\la_0 \to 0$ we obtain an infinitesimal necessary condition.
\begin{corollary}
If $h=(h_1,h_2,h_3)\in\hol(\D,\ccc)$ and $h(0)=(0,0,0)$ then
\[
 |h_1'(0)| \leq 1 \quad \mbox{ and }  \quad \half |h_2'(0)| +  |h_3'(0)| \leq 1.
\]
\end{corollary}

Is there a converse?  Is it the case that if
\beq\label{twoconds}
|A| \leq 1 \quad \mbox{ and }\quad \half|S|+|P| \leq 1
\eeq
then there exists $h\in\hol(\D,\bar\ccc) $ such that $h(0)=(0,0,0)$ and $h'(0)=(A,S,P)$?
The answer is no.  
\begin{example} \rm
Choose $A=1, \, 0<P<1$ and $S=2(1-P)$.    The inequalities \eqref{twoconds} hold.
Suppose there exists $h=(a,s,p) \in\hol(\D,\bar\ccc)$ with the required properties.  Since $a\in\schur, \, a(0)=0$ and $a'(0)=1$, Schwarz' Lemma asserts that $a(\la)= \la$ for $\la\in\D$.   Since $\half|S|+|P| = 1$ we know from \cite{geos} that there is a {\em unique}
function $(s,p) \in \hol(\D,\G)$ that maps $0$ to $(0,0)$ and has derivative $(S,P)$ at $0$, to wit
\[
(s,p)(\la) =\frac{\la}{1+P\la} (2(1-P), \la+P).
\]
However, the function $h(\la)=(\la,s(\la), p(\la))$ does not map $\D$ to $\bar\ccc$.  For $h(1)=(1,2\xi,1)$ where
$\xi= (1-P)/(1+P) \in (0,1)$.  For the point $(2\xi,1)$ we have $\beta=\xi$, and so
\[
\left| 1- \frac{\half s \bar\beta}{1+\sqrt{1-|\beta|^2}}\right|= 1-\frac{\xi^2}{1+\sqrt{1-\xi^2}} = \sqrt{1-\xi^2} <1.
\]
Hence $h(1)=(1,2\xi,1) \notin \bar\ccc$, which is a contradiction.
\end{example}

\section{Analytic lifting}\label{lifting}

In the present context the $\mu$-synthesis problem is an interpolation problem for analytic functions from $\D$ to $\B_\mu$.
If $H:\D\to\B_\mu$ is an analytic function satisfying interpolation conditions $H(\la_j) = W_j$ for given points $\la_1,\dots,\la_n\in\D$ and target points $W_1,\dots,W_n \in\B_\mu$, then $h\df\pi\circ H:\D\to \ccc$ is an analytic function that satisfies
\beq\label{piderived}
h(\la_j)=\pi(W_j) \quad \mbox{ for } j=1,\dots,n.
\eeq
The idea is that interpolation problems for $\hol(\D,\ccc)$ should be easier than those for $\hol(\D,\B_\mu)$, as the bounded $3$-dimensional domain $\ccc$ is likely to have a more tractable geometry than the unbounded $4$-dimensional domain $\B_\mu$.

If we can find $h\in\hol(\D,\ccc)$ satisfying the interpolation conditions \eqref{piderived}, does it follow that we can lift $h$ to a function $H\in\hol(\D,\B_\mu)$ that solves the original interpolation problem?  (For the analogous questions in the cases of the symmetrised bidisc and the tetrablock, the answer is roughly yes, though with a few technicalities).
We shall say that $H\in\hol(\D,\mat)$ is an {\em analytic lifting} of $h\in\hol(\D,\bar\ccc)$   if  $\pi\circ H=h$.  We say that $H$ is a {\em Schur lifting} of $h$ if $\pi\circ H=h$ and $H$ belongs to the matricial Schur class 
\[
\mathcal{S}_{2\times 2} \df \{F\in\hol(\D,\mat):  \|F(\la)\| \leq 1 \, \mbox{ for all } \la\in\D\}.
\]
 Of course, if $H$ is an analytic lifting of $h$ then $H\in\hol(\D,\bar\B_\mu)$ (see
 Corollary \ref{dependsonly}).

The lifting problem for $\hol(\D,\PP)$ is delicate, as the following three examples show.
\begin{example} \label{ex1}\rm
Let $h(\la)= (\la,0,\la)$.  This $h\in\hol(\D,\PP)$  lifts to $H\in\mathcal{S}_{2\times 2}$ given by
\[
H(\la)=\bbm 0&-1\\ \la&0 \ebm.
\]
\end{example}
Here $H(\la)$ does not belong to the open matrix ball $\B$ for any $\la\in\D$.  Our construction in Proposition \ref{asuffCondn} above gives the following non-analytic lifting of $(\la,0,\la)\in\PP$ to $\B$:
\[
H(\la)= \bbm \ii(1-|\la|)^\half\zeta & -|\la| \\ \la & -\ii(1-|\la|)^\half \zeta \ebm
\]
where $\zeta$ is a square root of $\la$.

\begin{example}\label{ex2} \rm
Let $h(\la)=(\la^2,0,\la)$.  Then $h\in\hol(\D,\PP)$, but there is no $H\in\hol(\D,\mat)$ such that $h=\pi\circ H$.  
\end{example}
For suppose $H$ has this property.  We can write 
\[
H=\bbm -\eta & g \\ \la^2 & \eta \ebm
\]
for some $g, \eta$ in $\hol \D$.  Since $\det H= \la$ we must have
\[
\eta(\la)^2=-\la-\la^2g(\la)
\]
for $\la\in\D$.  This is a contradiction, since the right hand side has a {\em simple} zero at $0$, while the left hand side has a zero of multiplicity at least $2$.

These examples point to Proposition \ref{liftprop}. To prove this proposition we will need the following lemma.

\begin{lemma}\label{even-zeros} 
Let $f_1, f_2 \in \schur$ be such that there is no $\alpha \in \D$ for which, for some odd positive integer $n$, $\alpha$ is a zero of $f_1$ of multiplicity $n$ and a zero of $f_2$ of multiplicity greater than $n$. Then there exists $g\in\hol\D$ such that $f_1 + f_2 g$ has no zeros of odd multiplicity in $\D$. 
\end{lemma}
\begin{proof} Here is a sketch of the proof.
Let $\{\alpha_j, j =1,2, \dots\}$ be the common zeros of $f_1$ and $f_2$. Under the hypothesis about the  orders of the  $\alpha_j$, it is easy to see that there is a Blaschke product $\phi$ whose zeros are the $\alpha_j, j =1,2, \dots$ and  there is a finite set  $I(\alpha_j)$ of interpolation conditions at each $\alpha_j$ such that every $g\in\hol\D$ which satisfies the conditions  $I(\alpha_j)$ at all $\alpha_j, j =1,2, \dots$ has the property that $f_1 + f_2 g = \phi^2 u $ for some $u\in\hol\D$ satisfying $u(\alpha_j) \neq 0$ for each $j$.

Let $\{\beta_i, i =1,2, \dots\}$ be the zeros of $f_2$ which are not zeros of $f_1$. We wish to choose $\gamma \in \hol\D$ such that 
\beq\label{defg}
g = \frac{\phi^2 \mathrm{e}^{\gamma} - f_1}{f_2}
\eeq
has the required property.
The condition that $g$ has a removable singularity at each $\beta_i$ is equivalent 
to a finite set $J(\beta_i)$ of interpolation conditions on $\gamma$ at $\beta_i$, while the condition that $g$ given by equation \eqref{defg} satisfies $I(\alpha_j)$ at each $\alpha_j$ yields a finite set of interpolation conditions on $\gamma$ at each $\alpha_j$.
Since, by \cite[Theorem 15.15]{rudin}, we may always find a $\gamma \in \hol\D$ satisfying a finite set of interpolation conditions at every point of $\{\alpha_j, j =1,2, \dots\} \cup \{\beta_i, i =1,2, \dots\}$, we obtain $g\in\hol\D$ such that $f_1 + f_2 g$ has zeros of even  multiplicity at all $\alpha_j$ and no zeros in $\D \setminus \{\alpha_j, j =1,2, \dots\}$. 
\end{proof}

\begin{proposition}\label{liftprop}
A function $h=(a,s,p)$ lifts to $\hol(\D,\mat)$ if and only if there is no point $\al\in\D$ such that, for some odd positive integer $n$,
\begin{enumerate}
\item $\al$ is a zero of $\tfrac 14 s^2-p$ of multiplicity $n$ and
\item $\al $ is a zero of $a$ of multiplicity greater than $n$.
\end{enumerate}
\end{proposition}
\begin{proof}
A function 
\beq\label{formH}
H= \bbm \half s-\eta& g \\ a & \half s+\eta \ebm
\eeq
is a lifting of $h=(a,s,p)\in\hol(\D,\PP)$ to $\hol(\D,\mat)$ if and only if $\eta, g \in\hol \D$ and $\det H = p$, that is,
\beq\label{condeta}
\eta^2= \tfrac 14 s^2 -p - ga.
\eeq

Suppose that $\al\in\D$ satisfies (1) and (2).  Then $\al$ is a zero of the right hand side of equation \eqref{condeta} of odd multiplicity $n$, whereas $\al$ is a zero of $\eta^2$ of even multiplicity.  This is a contradiction, and so necessity holds in Proposition \ref{liftprop}.

Conversely, suppose that there is no $\al\in\D$ such that (1) to (2) hold. 
Apply Lemma \ref{even-zeros} with $f_1= \tfrac 14 s^2-p$ and $f_2=-a$ to obtain
 $g\in\hol\D$ such that $\tfrac 14 s^2-p-ga$ has no zeros of odd multiplicity in $\D$ and hence has an analytic square root $\eta$. Then  $H$ of equation \eqref{formH} is the required lifting of $h$.
\end{proof}

There are functions $h\in\hol(\D,\bar\ccc)$ that have an analytic lifting but no Schur lifting.
\begin{example}\label{noSchurLifting} \rm  The function $h(\la)=(\half,0,\la)\in\hol(\D,\bar\ccc)$ has an analytic lifting but no Schur lifting.   More generally, let $a\in\Delta\setminus\{0\}$  and let $\ph,\psi$ be inner functions.  The function $h=(a\psi,0,\ph)\in\hol(\D,\bar\ccc)$ has an analytic lifting provided there is no point $\al\in\D$ that is a common zero of $\ph,\psi$ and has odd multiplicity $n$ for $\ph$ and multiplicity greater than $n$ for $\psi$.  However $h$ has a Schur lifting if and only if $\ph$ has an analytic square root and $\psi$ divides $\ph$ in $H^\infty$.   
\end{example}
\begin{proof}
The statement about the existence of an analytic lifting of $h$ follows from Proposition \ref{liftprop}. 

Suppose that $\ph=\ups^2$ for some inner function $\ups$ and $\psi$ divides $\ph$.  Then the function
\[
H=\bbm \ii (1-|a|^2)^\half \ups & -\bar a \ph /\psi\\ a\psi &-\ii (1-|a|^2)^\half \ups \ebm
\]
is a Schur lifting of $h$.

Conversely, suppose that $h$ has a Schur lifting $H$.   Necessarily $H$ has the form
\[
H=\bbm \eta& -(\eta^2+\ph)/(a\psi) \\ a\psi & -\eta \ebm
\]
for some $\eta$ in the Schur class $\schur$.  Since $\det (1-H^*H)\geq 0$ on $\Delta$,
\[
1-|a\psi|^2-2|\eta|^2 -\frac{|\eta^2+\ph|^2}{|a\psi|^2} + |\ph|^2\geq 0.
\]
Let $f=\eta^2\in\schur$.  Since $|f-\ph|\geq \left| |f|-|\ph|\right|$  and $\ph,\psi$ are inner, we have, a.e. on $\T$,
\[
2-|a|^2-2|f| -\frac{(|f|-1)^2}{|a|^2} \geq 0.
\]
This inequality simplifies to 
\[
0\geq (|f|+|a|^2-1)^2.
\]
It follows that $|f|=1-|a|^2$ a.e. on $\T$, and moreover all the inequalities in the sequence above are actually equalities.  In particular, $|f-\ph|^2= (|f|-|\ph|)^2$ and so
\[
\re(\bar\ph f) =-|f|= -(1-|a|^2) \quad\mbox{ a.e. on } \T
\]
and consequently
\[
-\bar\ph f= |f|=1-|a|^2 \quad \mbox{ a.e. on } \T.
\]
Thus 
\[
\eta^2=f=-(1-|a|^2)\ph
\]
 and so $\ph$ has an analytic square root.  Moreover $\eta^2+\ph=|a|^2\ph$, and so 
\[
-\bar a \ph/\psi = H_{12} \in \schur.
\]
Thus $\psi$ divides $\ph$.
\end{proof}
The upshot of Proposition \ref{liftprop} and the three examples is that the $\mu$-synthesis problem  for $\mu_E$ and the interpolation problem for $\hol(\D,\bar\ccc)$ are quite closely related, but that the rich function theory of $\hol(\D,\bar\B)$ may not be helpful for their solution.

\section{Conclusions} \label{conclud}

The genesis of this paper was an attempt to find a new case of the notoriously difficult $\mu$-synthesis problem that is amenable to analysis.  The $\mu$-synthesis problem arises in $H^\infty$ control theory, for example, in the problem of designing a robustly stabilising controller for plants which are subject to structured uncertainty \cite{Do,DuPa}.  Here $\mu$ denotes a cost function on the space of $m\times n$ complex matrices; as in Section \ref{instance}, it is given by 
\beq\label{defmu2}
\frac{1}{\mu_E(A)} = \inf \{\|X\|: X\in E \mbox{ and } \det (1-AX)=0\}
\eeq
where $E$ is a linear space of matrices of appropriate size.
 Previous attempts to find analysable instances of $\mu$-synthesis have led to the study of two domains in $\C^2$ and $\C^3$, the symmetrised bidisc $\G$ of Section \ref{Gamma} and the {\em tetrablock} (see for example \cite{awy,Y10}).  These domains have turned out to have interesting function-theoretic \cite{AY01,MRZ,PS}, operator-theoretic \cite{AY99,AYmodel,BPSR,Bhat,Sarkar} and geometric properties \cite{costara,AY04,EZ,JP,Z13}.  Could there be a class of `$\mu$-related domains' which have similarly rich theories, and which would throw light on the $\mu$-synthesis problem?  In this paper we study the next natural case of $\mu$, which results from taking the space $E$ in equation \eqref{defmu2} to be the space of $2\times 2$ matrices spanned by the identity matrix and a Jordan cell.  This choice leads to the pentablock $\ccc$.   As we have shown, $\ccc$ is indeed amenable to analysis, though there remain some fundamental questions about $\ccc$.  We list some of them below.

The $\mu$-synthesis problem is an interpolation problem for the space $\hol(\D,\Omega)$ for certain domains $\Omega\subset \C^d$.  One is given distinct points $\la_1,\dots,\la_N\in\D$ and target points $w_1,\dots,w_N\in\Omega$ and the task is to determine whether there exists $F\in\hol(\D,\Omega)$ such that $F(\la_j)=w_j$ for $j=1,\dots, N$, and if so to find such an $F$ (actually the interpolation conditions in \cite{Do,DuPa} are of a more general form).  In the case that $N=2$ this problem is central to hyperbolic geometry in the sense of Kobayashi \cite{Ko98}, so one could describe the problem as belonging to hyper-hyperbolic geometry.   In $\mu$-synthesis the domain $\Omega$ has the form
\[
\Omega_\mu=\{A\in \C^{m\times n}: \mu(A)<1\}.
\]
  This is typically an unbounded nonconvex and hitherto unstudied domain, and so the construction of holomorphic maps from $\D$ to $\Omega_\mu$ is a challenge.  In the cases that $\mu$ is the spectral radius and $\mu_{\mathrm {diag}}$ there is an effective technique of dimension-reduction.  

Let us say that the {\em polynomial rank} of a domain $\Omega\subset \C^d$ is the smallest positive integer $r$ such that there exists a polynomial map $\pi:\C^d\to\C^r$ and a domain $\Omega'\subset\C^r$ such that $z\in\C^d$ belongs to $\Omega$ if and only if $\pi(z)\in \Omega'$.  More succinctly, $\pi$ must satisfy $\Omega=\pi\inv(\pi(\Omega))$.  Clearly $r\leq d$, since we may choose $\pi$ to be the identity map on $\C^d$. In contrast, in all the special cases of $\mu$ mentioned in this paper it turns out that the polynomial rank of $\Omega_\mu$ is {\em less than} the dimension of the domain.  In particular, Corollary \ref{dependsonly} shows that the polynomial rank of $\Omega_{\mu_E}$ is at most $3$.
The idea is that, when the polynomial rank of $\Omega$ is less than its dimension, the geometry of the lower-dimensional domain may be more accessible than that of $\Omega$ itself.  A strategy for the construction of interpolating functions from $\D$ to $\Omega$ is to find a map $h\in\hol(\D,\pi(\Omega))$ which satisfies $h(\la_j)=\pi(w_j)$ for each $j$, and then to attempt to lift $h$ modulo $\pi$ to an interpolating function in $\hol(\D,\Omega)$.

When $\Omega=\Omega_\mu$ for some $\mu$ the problem has a further helpful feature: since $\mu_E$ is no greater than the operator norm, for any subspace $E$, it is always the case that $\Omega_\mu$ contains the open unit ball of the ambient space of matrices.  In all three of the special cases of interest it turns out that the images of $\Omega_\mu$ and the unit ball $\B$ under the dimension-reducing map $\pi$  {\em coincide}.   Now the geometry and function theory of the Cartan domain $\B$ is rich and long established, and there are numerous ways of constructing maps in $\hol(\D,\B)$; for example one may use the homogeneity of $\B$ to construct an interpolating function $H$ by the standard process of Schur reduction.    Then $\pi\circ H$ is a holomorphic function from $\D$ to $\pi(\B)$ satisfying interpolation conditions, and one may then try to find an analytic lifting of $\pi\circ H$ to an element of $\hol(\D,\Omega_\mu)$ that satisfies the given interpolation conditions.  This strategy has had some successes, admittedly modest, for the two special cases of $\mu$ mentioned above.

In this new case of $\mu$ the strategy again looks promising. The dimension-reducing map $\pi$ here takes $A\in\mat$ to $(a_{21}, \tr A,\det A)$, and Theorem \ref{4equiv} shows that $\pi\inv(\pi(\B))=\B_\mu$.  Here $\pi(\B)$ is the pentablock and we write 
$\B_\mu$ rather that $\Omega_\mu$.    The strategy outlined above is in principle feasible.  However, Sections \ref{schwarz} and  \ref{lifting} show that the final step, the lifting of maps from $\hol(\D,\ccc)$ to $\hol(\D,\B_\mu)$ is more subtle than in previous cases.  

  We end with two natural questions.

Do the Carath\'eodory distance and Lempert functions coincide on the pentablock?  See \cite{EdKoZw10} for a positive solution of the corresponding question for the tetrablock.

What are the magic functions of the pentablock?  See \cite{AY08} for the definition of magic function and for their use in determining the automorphisms of a domain.

In the original version of this paper at arXiv:1403.1960 we also asked whether the pentablock is an analytic retract of $\B$. 
It has now been shown \cite{kos} that the answer is negative, as in the corresponding  question for the tetrablock \cite{Y08}. It follows that the pentablock is inhomogeneous.

\bibliographystyle{plain}

\begin{thebibliography}{30}\label{pentablock_bbl}

\bibitem{awy} A. A. Abouhajar, M. C. White and N. J. Young, A Schwarz lemma for a domain related to mu-synthesis, {\em  J. Geom. Anal.} {\bf 17} (2007) 717--750.

\bibitem{AY99} J. Agler and N. J. Young, A commutant lifting theorem for  a domain in ${\mathbb C}^2$ and spectral interpolation, {\em J.  Functional Analysis} {\bf 161} (1999) 452--477.
\bibitem{AY01}
J. Agler and N. J. Young,  A Schwarz lemma for the symmetrised bidisc,
  {\em Bull. London Math. Soc.} \textbf{33} (2001) 175--186.

\bibitem{AYmodel} J. Agler and N. J. Young, A model theory for $\Gamma$-contractions, {\em J. Operator Theory} {\bf 49} (2003) 45-60.

\bibitem{AY04}  J. Agler and N. J. Young, The hyperbolic geometry of the symmetrized bidisc, {\em J. Geom. Anal.} {\bf 14} (2004) 375--403.

\bibitem{geos} J. Agler and N. J. Young, The complex geodesics of the symmetrized bidisc, {\em International J. Math.} {\bf 17} (2006) 375-391.

\bibitem{AY08} J. Agler and N. J. Young, The magic functions and automorphisms of a domain, {\em Complex Anal. Operator Theory} {\bf 2} (2008) 383-404.



\bibitem{Bhat} T. Bhattacharyya, The tetrablock as a spectral set, to appear in {\em Indiana Univ. Math. J.}

\bibitem{BPSR} T. Bhattacharyya, S. Pal and S. Shyam Roy, Dilations of $\Gamma$-contractions by solving operator equations,
 {\em Adv. in  Math.}, {\bf 230}  (2012) 577-606.

\bibitem{browder}
A. Browder,  {\em Introduction to function algebras,}  W. A. Benjamin Inc., New York, 1969.

\bibitem{clerc}  J-L. Clerc, Geometry of the Shilov boundary of a bounded domain, {\em J. Geometry and Symmetry in Physics} {\bf 13} (2009) 25-74.

\bibitem{costara} C. Costara, The symmetrized bidisc and Lempert's theorem, {\em Bull. London Math. Soc.} {\bf 36} (2004) 656--662. 

\bibitem{Do} J. C. Doyle, Analysis of feedback systems  with structured  uncertainties. {\em  IEE Proceedings} {\bf 129} (1982), no. 6, 242--250.

\bibitem{DuPa} 
G. Dullerud and F. Paganini, {\em A course in robust control theory: a convex approach}, Texts in Applied Mathematics {\bf 36}, Springer (2000).

\bibitem{EdKoZw10}
A. Edigarian, L. Kosinski and W. Zwonek, The Lempert theorem and the tetrablock,  {\em J. Geom. Anal. } {\bf 23} (4) (2013) 1818-1831.

\bibitem{EZ}
A. Edigarian and W. Zwonek, Geometry of the symmetrised polydisc, {\em Archiv  Math.}, {\bf 84}  (2005) 364-374.

\bibitem{JP}
M. Jarnicki and P. Pflug, On automorphisms of the symmetrised bidisc,
 {\em Arch. Math. (Basel)}  {\bf 83}  (2004), no. 3, 264--266.


\bibitem{Ko98} S. Kobayashi, {\em Hyperbolic complex spaces}, Springer, New York, 1998.

\bibitem{kos}  L. Kosi\'nski, The group of automorphisms of the pentablock, arXiv:1403.5214, March 2014.

\bibitem{krantz}
S. G. Krantz, {\em Function theory of several complex variables},  John Wiley and Sons,  New York, 1982.

\bibitem{MRZ} G. Misra, S. S. Roy and G. Zhang, Reproducing kernel for a class of weighted Bergman spaces on the symmetrized polydisc, {\em Proc. Amer. Math. Soc.} {\bf 141} (2013)  2361-2370 .

\bibitem{NiPfZw}
N. Nikolov, P. Pflug and W. Zwonek, The Lempert function of the symmetrized polydisc in higher dimensions is not a distance,
 {\em Proc. Amer. Math. Soc.} {\bf 135} (2007) 2921--2928.

\bibitem{PS}  S. Pal and O. Shalit,    Spectral sets and distinguished varieties in the symmetrized bidisc, {\em J. Functional Analysis},  {\bf 266} (9) (2014) 5779--5800.

\bibitem{rudin} W. Rudin, {\em Real and complex analysis}, McGraw-Hill, New York, 1966.

\bibitem{Sarkar}  J. Sarkar, Operator theory on symmetrized bidisc,  (2012) arXiv:1207.1862 .

\bibitem{Y08}
N. J. Young, The automorphism group of the tetrablock, {\em J. London Math. Soc.} {\bf 77}  (2008) 757-770.

\bibitem{Y10}  N. J. Young,   Some analysable instances of $\mu$-synthesis,   {\em Mathematical methods in systems, optimization and control}, Editors: H. Dym, M. de Oliveira, M. Putinar, Operator Theory: Advances and Applications {\bf 222}  349--366, Birkh\"{a}user, Basel, 2012.

\bibitem{Z13} W. Zwonek, Geometric properties of the tetrablock, {\em Archiv der Mathematik} {\bf 100} (2013) 159-165.\\
\end{thebibliography}

JIM  ~ AGLER, Department of Mathematics, University of California at San Diego, CA \textup{92103}, USA\\

ZINAIDA A. LYKOVA,
School of Mathematics and Statistics, Newcastle University, Newcastle upon Tyne
 NE\textup{1} \textup{7}RU, U.K.~~\\

N. J. YOUNG, School of Mathematics and Statistics, Newcastle University, Newcastle upon Tyne NE1 7RU, U.K.
 {\em and} School of Mathematics, Leeds University,  Leeds LS2 9JT, U.K.
\end{document}